\documentclass[11pt,a4paper]{article}

\RequirePackage[numbers]{natbib}
\RequirePackage[colorlinks,citecolor=blue,urlcolor=blue]{hyperref}

\usepackage{fullpage,lmodern,bm}
\usepackage[T1]{fontenc}
\usepackage[utf8]{inputenc}                        
\usepackage[Lenny]{fncychap}
\usepackage[usenames,dvipsnames]{xcolor}
\usepackage{amsmath,amssymb,amsthm,thmtools,dsfont,enumerate}
\usepackage{tikz,epigraph,lipsum}
\usetikzlibrary{shapes,arrows}
\usepackage{rotating,lipsum,pifont,cases}
\usepackage{appendix,authblk}
\usepackage{cases}
\usepackage{pdfsync}
\usetikzlibrary{patterns,intersections}

\definecolor{blue0}{RGB}{0,77,153} % dark blue
\definecolor{red0}{RGB}{179,0,77} % dark blue
\definecolor{green0}{RGB}{134,219,76} % dark blue
\definecolor{gray0}{RGB}{84,97,110}% gray blue

\numberwithin{equation}{section}

\newtheorem{theorem}{Theorem}[section]
\newtheorem{proposition}[theorem]{Proposition}
\newtheorem{lemma}[theorem]{Lemma}
\newtheorem{definition}[theorem]{Definition}
\newtheorem{remark}[theorem]{Remark}
\newtheorem{example}[theorem]{Example}

\DeclareMathOperator{\Tr}{Tr}
\DeclareMathOperator{\Ker}{Ker}

\DeclareMathOperator{\diag}{diag}
\DeclareMathOperator{\Conv}{Conv}

\DeclareMathOperator{\vecop}{vec}

\newcommand{\calC}{\ensuremath{\mathcal{C}}}
\newcommand{\calD}{\ensuremath{\mathcal{D}}}
\newcommand{\calF}{\ensuremath{\mathcal{F}}}
\newcommand{\calL}{\ensuremath{\mathcal{L}}}
\newcommand{\calN}{\ensuremath{\mathcal{N}}}

\newcommand{\bbE}{\ensuremath{\mathbb{E}}}

\newcommand{\bbR}{\ensuremath{\mathbb{R}}}
\newcommand{\bbS}{\ensuremath{\mathbb{S}}}

\numberwithin{equation}{section}
\def\vs#1{\vspace{#1mm}}

\def\be{\begin{align}}
\def\ee{\end{align}}
\def\b*{\begin{eqnarray*}}
\def\e*{\end{eqnarray*}}

\def\be{\begin{eqnarray}}
\def\ee{\end{eqnarray}}
\def\beq{\begin{equation}}
\def\eeq{\end{equation}}
\def\b*{\begin{eqnarray*}}
\def\e*{\end{eqnarray*}}
\def\bi{\begin{itemize}}
\def\ei{\end{itemize}}

\def \1{{\bf 1}}

\def\eps{\varepsilon}

\def\={\;=\;}

\def\x{\times}

\def\diag#1{\mbox{\rm diag}\left[#1\right]}

 \def\reff#1{{\rm(\ref{#1})}}

 \def\vs#1{\vspace{#1mm}}

\def \F{\mathbb{F}}

\def \M{\mathbb{M}}

\def \R{\mathbb{R}}

\def\Cc{{\cal C}}
\def\Dc{{\cal D}}

\def\Lc{{\cal L}}

\def\Nc{{\cal N}}

\def\Lc{{\cal L}}

\newcommand{\inwardoutward}{%
	\begin{minipage}{.4\textwidth}
		\begin{center}
			\begin{tikzpicture}
			\draw [rounded corners] (-2,-1.7*1.7) rectangle (2,1.7*1.7);
			\draw[blue0, pattern color=lightgray, pattern=north west lines, ultra thick, domain=-1.7:1.7] plot (\x, {\x*\x});
			\draw [<->] (0,1.7*1.7)--(0,-1.7*1.7) -- (0,0)--(-2,0) -- (2,0);
			\draw[red0, ultra thick] [<-] (0,1) -- (0,0);
			\node at (-0.8,2.6) {$\{ \tilde x \geq \bar x^2\}$}; 
			\node at (0,3.3) {(i)  \textit{Affine diffusion}}; 
			\node at (0,1.2) {\textcolor{red0}{$b(0)$}}; 
			\end{tikzpicture}
		\end{center}
	\end{minipage}
	\begin{minipage}{.4\textwidth}
		\begin{center}
			\begin{tikzpicture}
			\draw [ pattern color=lightgray, pattern=north west lines,rounded corners] (-2,-1.7*1.7) rectangle (2,1.7*1.7);
			\draw[blue0, fill=white, ultra thick, domain=-1.7:1.7] plot (\x, -{\x*\x});
			\draw [<->] (0,1.7*1.7)--(0,-1.7*1.7) -- (0,0)--(-2,0) -- (2,0);
			\draw[red0, ultra thick] [<-] (0,-1) -- (0,0);
			\node at (-1,2.6) {$\{ \tilde x  \geq -\bar x^2\}$}; 
			\node at (0,3.3) {(ii)  \textit{Polynomial diffusion}}; 
			\node at (0,-1.2) {\textcolor{red0}{$b(0)$}}; 
			\end{tikzpicture}
		\end{center}
	\end{minipage}  
}%

\title{Stochastic invariance of closed sets with non-Lipschitz coefficients}
\date{\today}

\author[1,2]{Eduardo Abi Jaber \thanks{abijaber@ceremade.dauphine.fr,  {eduardo.abijaber@axa-im.com}}}
\author[1]{Bruno Bouchard\thanks{bouchard@ceremade.dauphine.fr. I would like to thank Christa Cuchiero,  {Josef} Teichmann and Damir Filipović for very fruitful discussions and comments.}}
\author[2]{Camille Illand \thanks{camille.illand@axa-im.com}}
\affil[1]{Universit\'e Paris-Dauphine, PSL  University, CNRS, CEREMADE, 75016 Paris, France.}
\affil[2]{AXA Investment Managers,  Multi Asset Client Solutions, Quantitative Research, \break
                        6 place de la Pyramide, 92908 Paris - La Défense, France.}

\begin{document}

	\maketitle

	\begin{abstract}

		This paper provides a new characterization of the  stochastic invariance of a  closed subset of $\mathbb{R}^d$ with respect to a diffusion. We extend the well-known inward pointing Stratonovich drift condition to the case where the diffusion matrix can fail to be differentiable: we only assume that the covariance matrix is.  In particular, our result can be applied to construct  affine   and   polynomial diffusions on any arbitrary closed set.\\ \\

		
		
	\end{abstract}

\section{Introduction}
%:
Let $b:\mathbb{R}^d \mapsto \mathbb{R}^d$ and $\sigma:\mathbb{R}^d \mapsto \mathbb{M}^d$  be continuous functions, where $ \mathbb{M}^d$  denotes the space of   $d \times d$ matrices.  We assume that $b$ and $\sigma$ satisfy the following linear growth conditions: there exists $L >0 $   such that 
\begin{equation}\label{growthconditions}
\|b(x)\| + \|{\sigma\sigma^{\top}(x)}\|^{\frac{1}{2}} \leq L (1+ \|x\|), \quad \forall\; x \in \mathbb{R}^d, \tag{$H_1$} 
\end{equation}
  and we consider a weak solution of the stochastic differential equation 
\begin{equation}\label{diffusionsdeinvariance}
dX_t=b(X_t)dt+\sigma(X_t)dW_t, \quad X_0=x,
\end{equation}
i.e.~a $d$-dimensional Brownian motion $W$ and an adapted process $X$ such that the above equation holds. 
\vs2

The aim of  this paper is to provide a characterization of  the \textit{stochastic invariance} of {a} closed set $\mathcal{D}\subset \R^{d}$,  {i.e.}~find   necessary and sufficient conditions on the instantaneous drift $b$ and the instantaneous covariance matrix $\sigma\sigma^{\top}$ under which there exists a  {weak solution} of \eqref{diffusionsdeinvariance} that remains in $\mathcal{D}$ for all $t \geq 0$, almost surely, given that $x \in \mathcal{D}$. (See Definition \ref{def: stock inv} below for a precise formulation.)\\

The first  {stochastic invariance} results can be found in {Stroock and Varadhan \cite{SV72}}, Friedman \cite{fri} and Doss \cite{dos} . Since then, many extensions were considered in the literature.  For an arbitrary closed set, the  {stochastic invariance} was characterized through the second order normal cone in Bardi and Goatin \cite{barg} and Bardi and Jensen \cite{barj}. Aubin and Doss \cite{aubd} used the notion of  curvature, while Da Prato and Frankowska \cite{da} provided a characterization in terms of the Stratonovich drift. For a closed convex set, the distance function was used in Da Prato and Frankowska \cite{daf1},  and the invariance was characterized for  {affine jump-diffusions} in Tappe \cite{tap}.

Although these approaches  {differ}, they have at least one thing in common: the tradeoff one has to make between the assumptions on the topology/smoothness of the domain and the regularity of the coefficients $b$ and $\sigma$. This makes all of these existing results difficult to apply in practice. Let us start by highlighting this difficulty through the two main contributions to  the literature:
\begin{enumerate}[(i)]
	\item 
	In Bardi and Jensen \cite{barj}{,} the {stochastic invariance} is characterized by using Nagumo-type geometric conditions on the second order normal cone. Their main result states that the closed set $\mathcal{D}$ is  {stochastically invariant} if and only if
 \begin{equation*}\label{eqbardi}
 u^\top b(x) + \frac{1}{2} \Tr(vC(x)) \leq 0, \;\forall\; x \in \mathcal{D} \mbox{ and } (u,v) \in \mathcal{N}^2_{D}(x), 
 \end{equation*}
 in which {$C:=\sigma\sigma^{\top}$} on $\Dc$ and $ \mathcal{N}^2_{D}(x)$ is the second order normal cone at the point $x$: 
\begin{equation}\label{eq: def second order cone}	
\mathcal{N}^2_{\mathcal{D}}(x) :=\left\{(u,v) \in \mathbb{R}^d \times \mathbb{S}^d: \langle u, y-x \rangle + \frac{1}{2} \langle y-x , v (y-x) \rangle\leq o(\|y-x\|^2), {\forall\; y \in \calD} \right\}.
\end{equation}
Here, $\mathbb{S}^d$ {stands for} the cone of symmetric {$d\times d$} matrices. 
In practice, we  face  two restrictions. Prior to deriving the conditions on $b$ and $\sigma$, we have to determine the second order normal cone at all points of a given set. When the boundary is smooth, the computation of the second order normal cone is an easy task, see e.g.~\cite[Example 1]{barj}. However, it is much more challenging in general, by lack of efficient techniques. 
This renders the result of \cite{barj}   difficult to use in practice.  {This also corresponds to the positive maximum principle of  Ethier and  Kurtz \cite{eth}.} 
\item
Building on   Doss \cite{dos},  Da Prato and Frankowska \cite{da} {give} necessary and sufficient conditions for  the  stochastic invariance in terms of  the Stratonovich drift and the first order normal cone:
	\begin{equation}\label{dapratocondintro}
	\sigma(x)^\top u =0   \mbox{ and }  
	\langle u, b(x)-\frac{1}{2} \sum_{j=1}^{d} D\sigma^j(x)\sigma^j(x)  \rangle \leq 0, \;\forall\; x \in \mathcal{D} \mbox{ and }u \in \mathcal{N}^1_{\mathcal{D}}(x),
	\end{equation}	  
	where $\sigma^j(x)$ denotes the $j$-th column of the matrix $\sigma(x)$, $D\sigma^j$ is the Jacobian of $\sigma^j$, and the first order normal cone   $\mathcal{N}^1_{\mathcal{D}}(x)$ at $x$ (sometimes simply called \emph{normal cone}) is defined as
\begin{equation}\label{eq: def first order cone}
		\mathcal{N}^1_{\mathcal{D}}(x) :=\left\{u \in \mathbb{R}^d: \langle u, y-x \rangle \leq o(\|y-x\|), {\forall\; y \in \calD} \right\}.
\end{equation}
{In practice, the first order normal  cone is much  simpler to compute than the second order {cone} used in \cite{barj}}, see \cite{aubf}  and \cite{roc}. However, the price to pay is to impose a strong regularity {condition} on the diffusion matrix $\sigma$, which is assumed to be bounded and differentiable on $\mathbb{R}^d$, with a bounded Lipschitz derivative. Again, this   constitutes a sticking point for  applications, it {cannot be applied} to simple models (think about {square-root} processes for instance, see below).
\end{enumerate}

The aim of the present paper is   to extend the characterization \eqref{dapratocondintro}, given in terms of the \emph{easy-to-compute}  first order normal cone,   under weaker regularity conditions on the diffusion matrix $\sigma$. We make the following seemingly {trivial} observation: ${C:=\sigma \sigma^\top}$ might be differentiable at a point $x$ {while} $\sigma$ is not. It is the case for the {square-root process  mentioned above,} {at the boundary point $x=0$}. Moreover, the terms $ D\sigma^j(x)\sigma^j(x)$ can be rewritten in terms of   the Jacobian of $C$ whenever both quantities are well defined, see Proposition \ref{propequivalencedaprato} for a precise formulation. {This suggests to reformulate \eqref{dapratocondintro} with the Jacobian matrices of the columns of $C$ instead of $\sigma$.} \vs2

We prove that this is actually possible. Our main result, Theorem \ref{MainTheorem} below, states that the  {stochastic invariance} is   equivalent to the following conditions: 
	\begin{equation}\label{eq: nec suff cond intro}
	C(x) u =0   \mbox{ and }  
	\langle  u, b(x)-\frac{1}{2} \sum_{j=1}^{d} D C^j(x)(CC^+)^j(x)    \rangle \leq 0, \;\forall\; x \in \mathcal{D} \mbox{ and } u \in \mathcal{N}^1_{\mathcal{D}}(x).
	\end{equation}
 Here,  $(CC^+)^j(x)$ is the $j$-th column of $(CC^{+})(x)$ with $C(x)^+$ defined as the Moore-Penrose pseudoinverse of $C(x)$, see Definition \ref{defpseudoinverse} in the Appendix.  {We    only assume that 
\begin{equation} \label{eq: extension C}
\mbox{$C$ can be extended to a $\mathcal{C}^{1,1}_{loc}(\mathbb{R}^d,\mathbb{S}^d)$ function that coincides with $\sigma\sigma^{\top}$ on $\Dc$,}\tag{$H_2$}
\end{equation}
in which $\mathcal{C}^{1,1}_{loc}$ means $\Cc^{1}$ with {a} locally Lipschitz {derivative}. Note that we do not impose the extension of $C$ to be positive semi-definite outside $\Dc$, so that $\sigma$ might only match with its square-root on $\Dc$. Also, it should be clear that the extension {needs} only to be local around $\Dc$.}
\vs2 

The term $CC^+$ in \eqref{eq: nec suff cond intro} plays the role of {the} projection on the image of $C$, see Proposition \ref{proppseudoinverseproj} in the Appendix and the discussion in Remark \ref{rem : projection} below. This projection term {cannot} be removed. To see this,  let us consider the square-root process with  $C(x)=x$ and   $\mathcal{D}={\mathbb R}_{+}$, so that  {  $\mathcal{N}^1_{\mathcal{D}}(0)=\R_{-}$}. Then, 
 	\begin{equation*}
	C(0) (-1) =0 \quad \mbox{ and } \quad
	\langle  -1, b(0)-\frac{1}{2} D C(0)    \rangle \leq 0 
	\end{equation*}
	leads to $b(0)\ge 1/2$ while the correct condition {for invariance}  is $b(0)\ge 0$, which is recovered from \eqref{eq: nec suff cond intro} by using the fact that $(CC^{+})(0)=0$. 
	\\

This extension of the characterization of Da Prato and Frankowska \cite{da} provides for the first time a unified criteria for the case where the volatility matrix may not be $\Cc^{1}$ on the whole domain, which is of importance in practical situations. {In fact,} many models used in practice, in mathematical finance {for instance}, do not have $\Cc^{1}$ volatility maps but satisfy our conditions.  {This is in particular the case of   {affine diffusions}   ({see \cite{dfs,film}}), or of {polynomial  diffusions} that are characterized by a quadratic covariance matrix ({see \cite{cuchkr,fla}}), etc.}  When applied to such processes,  {stochastic invariance} results have  been so far tweaked  in order to fit in the previous set up, or {have been proved} under limiting conditions, on a case by case basis. For instance, in their construction of  {affine processes} on the cone of symmetric semi-definite matrices,    Cuchiero \textit{et al.}~\cite{cuchf} {start by regularizing} the martingale problem {before applying} the  {stochastic invariance characterization} of \cite{da} and then pass to the limit. In Spreij and Veerman \cite{sp12}, some  {stochastic invariance} results are also derived for  {affine diffusions} but only on  convex sets with smooth boundary. {More recently}, the mathematical foundation for  {polynomial  diffusions} is given in  Filipović and Larsson \cite{fla}. Necessary conditions for the  {stochastic invariance} are derived for basic closed semialgebraic sets. However, these conditions are not sharp,  their sufficient conditions differ from their necessary conditions. All the above cases can now be treated by using our characterization. See Section \ref{SectionExamples} for a generic  example.    
 \vs2

Our proof of the necessary condition is in  the spirit of Buckdahn \textit{et al.}~\cite{buc}. They use a second order stochastic Taylor expansion together with small time behavior results for double stochastic integrals. {However, in our case,} the stochastic Taylor expansion {cannot} be applied directly since $\sigma$ is not differentiable and $\sigma(X)$ fails to be a semi-martingale whenever an eigenvalue {vanishes} (see \cite[Example 1.2]{mij}). We therefore need to develop new ideas.  We first observe that, if $\sigma$ is diagonal, then vanishing eigenvalues can   be eliminated {by taking the} conditional expectation with respect to the path of the Brownian motion acting on the {non-vanishing} ones.   This corresponds to the projection term $CC^{+}$ in \eqref{eq: nec suff cond intro}.      If $\sigma$ is not diagonal, we can essentially reduce to the former case by considering its spectral decomposition and a suitable change of Brownian motion (based on the corresponding basis change), see Lemma \ref{lemmadistincteigen} below. {However, it requires}  a smooth spectral decomposition which is not guaranteed when repeated eigenvalues are present. To avoid this, we need an additional transformation of the {state} space, see Proposition \ref{propnecessity}. 
\vspace{2mm}

Conversely, we show that   the infinitesimal {generator} of our diffusion satisfies the \textit{positive maximum principle}  whenever \eqref{eq: nec suff cond intro} holds, see Section \ref{SectionSuf} below. {Applying    \cite[Theorem 4.5.4]{eth}} {shows} that this condition is indeed sufficient. (Note that the approach based on the comparison principle for viscosity solutions used in \cite{barj,buc} {cannot} be applied to our case since $\sigma$ is not Lipschitz.)
\\

The rest of the paper is organized as follows. Our main result is stated in Section \ref{SectionMain}. The proofs are collected in Sections \ref{SectionNec} and \ref{SectionSuf}.
In Section \ref{SectionExamples}, we exemplify {our characterization} by deriving explicit  stochastic invariance conditions for  various typical examples of application{s}. Finally, Section \ref{sectionboundarynon} provides a complementary  tractable sufficient {condition} ensuring {the stochastic invariance of the interior of a domain}.  
For the convenience of the reader, we collect some standard {results} of matrix calculus and differentiation in the Appendix. 
\vs2

From now on, all identities involving random variables have to be considered in the a.s.~sense, the probability space and the probability measure being given by the context. Elements of $\R^{d}$ are viewed as column vectors. The vector $e_{i}\in \R^{d}$ is the $i$-th element of the canonical basis, and we use the standard notation $I_{d}$ to denote the $d\times d$ identity matrix.  We denote by $\M^{d}$ the collection of $d\times d$ matrices. We say that $A\in \mathbb{S}^{d}$ (resp.~$\mathbb{S}^{d}_{+}$) if it is a symmetric (resp.~and positive semi-definite) element of $\M^{d}$.
Given $x=(x^{1},\ldots,x^{d})\in \R^{d}$, $\diag{x}$ denotes the diagonal matrix whose $i$-th diagonal component is $x^{i}$.  If $A$ is a symmetric positive semi-definite matrix, then $A^{\frac12}$ stands for its symmetric square-root.

%%%%%%%%%%%%%%%%%%%%%%%%%%%%%%%%%%%%%%%%%%%%%%%%%%%%%%%%%%%%%%%%%

\section{Main result}\label{SectionMain}

In this section, we state our main result, Theorem \ref{MainTheorem}, that extends Theorem 4.1 in Da Prato and Frankowska \cite{da} to weaker regularity assumptions. 

\vs2
Since we are dealing with general coefficients $b$ and $\sigma$,  {i.e.}~not necessarily Lipschitz coefficients, solutions to the  {stochastic differential equation} \eqref{diffusionsdeinvariance} should be considered in the  {weak} sense rather than in the  {strong} sense. Existence is guaranteed by our condition \eqref{growthconditions}, together with our standing assumption of  continuity of $b$ and  $\sigma$:  there exist a filtered probability space $(\Omega,\mathcal{F},{\F=}(\mathcal{F}_t)_{t \geq 0},\mathbb{P})$ satisfying the usual conditions, a $d$-dimensional $\F$-Brownian motion $W$ and a ${\F}$-adapted process $X$ with continuous sample paths such that   \eqref{diffusionsdeinvariance} holds $\mathbb{P-}$a.s. See e.g.~\cite[Theorems IV.2.3 and IV.2.4]{ike}. 
 
For later use, note that  \eqref{growthconditions} implies that, for any positive integer $p$, there exists $K_{p,x} > 0$ such that
\begin{equation}\label{holdertrajectories}
\mathbb{E}\left[ \|X_t-X_s\|^p \right] \leq K_{p,x} |t-s|^{\frac{p}{2}}
\end{equation}
 for all $0\le  s,t\le  1$. Hence, Kolmogorov's continuity criterion ensures that the sample paths of $X$ are (locally) $\eta$-H\"older continuous for any $\eta \in (0,\frac{1}{2})$ (up to considering a suitable modification).

\begin{remark}\label{rmkchoiceofsigma}
	The collection $\cal Q$ of possible distributions of $X$ is entirely determined by the infinitesimal generator $\mathcal{L}$ defined on the space of smooth functions $\phi$ by 
	$
	\Lc \phi:=D\phi \,b + \frac12 {\rm Tr}[\sigma\sigma^\top D^{2}\phi].
	$
Therefore, $\cal Q$ is the same if $\sigma$ is replaced by  $\tilde \sigma$ such that $\tilde \sigma\tilde \sigma^{\top}=\sigma\sigma^\top$, {see e.g.}~\cite[Remark 5.1.7]{str}.  Hence, we can reduce   to the case where $\sigma$ is the  {symmetric} square-root of $C$ on $\Dc$, which we will assume from now on.
\end{remark}
 
  Before   stating our main result, let us make precise the definition of stochastic invariance. 

\begin{definition}[\textit{Stochastic invariance}]\label{def: stock inv} A closed subset $\mathcal{D} \subset \mathbb{R}^d$ is said to be \textit{stochastically invariant} with respect to the diffusion \eqref{diffusionsdeinvariance} if, for all $x \in \mathcal{D}$, there exists a weak solution $(X,W)$ to \eqref{diffusionsdeinvariance} starting at $X_0=x$ such that $X_t \in \mathcal{D}$ for all $t\geq0$, almost surely. 
\end{definition}

Our characterization of stochastic invariance reads as follows (see Propositions \ref{propnecessity}   and   \ref{propsufficiency} below for the proof).  {From now on we use the same notation $C$ for $C$ defined as $\sigma\sigma^{\top}$ on $\Dc$ and for its extension defined in Assumption \eqref{eq: extension C}. }

\begin{theorem}[Invariance characterization]\label{MainTheorem}
	Let  $\mathcal{D}$ be closed. Assume that $b${, $\sigma$} and $C$ are continuous and satisfy assumptions \eqref{growthconditions}-\eqref{eq: extension C}.  Then, the set $\mathcal{D}$ is \textit{stochastically invariant} with respect to the diffusion \eqref{diffusionsdeinvariance} if and only if 
	\begin{subnumcases}{}
	C(x) u =0 \label{ourcondcov}
	\\
	\langle  u, b(x)-\frac{1}{2} \sum_{j=1}^{d} D C^j(x)(CC^+)^j(x)    \rangle \leq 0 \label{ourconddrift}
	\end{subnumcases}
	 for every $x \in \mathcal{D}$ and for all $u \in \mathcal{N}^1_{\mathcal{D}}(x)$.  
\end{theorem}

Clearly, the regularity conditions of Theorem \ref{MainTheorem} are much weaker than those of Theorem 4.1 in Da Prato and Frankowska \cite{da}.   Let us immediately exemplify this by considering  the case of the square-root process {already mentioned} in the introduction. Let $\mathcal{D}=\mathbb{R}_+$, $C(x)= \eta^2 x$ with $\eta > 0$, and consider the   diffusion  $dX_t= b(X_t) dt +  \eta \sqrt{X_t}dW_t$. Since $C(x)C(x)^+=\mathds{1}_{\{ x > 0 \}}$ and $\mathcal{N}^1_{\mathbb{R_+}}{(x)}=\mathds{1}_{\{x=0\}}\R_{-}$, Theorem \ref{MainTheorem} implies that $\R_{+}$ is stochastically invariant   if  and only if $b(0)\geq 0$, {while} $\sigma: x\in \R_{+}\mapsto \eta\sqrt{x}$ is not differentiable at $0$.

  On the other hand, one can easily recover \cite[Theorem 4.1]{da} under their smoothness assumptions. {This is the object of the next proposition (recall that, by Remark \ref{rmkchoiceofsigma}, the study can be reduced to the case $C=\sigma^2$ on $\Dc$)}.

\begin{proposition}\label{propequivalencedaprato}
	Fix  $\sigma \in \mathcal{C}^{1,1}_b(\mathbb{R}^d, {\mathbb{S}^d})$ (\textit{i.e.}~$\sigma$ is  differentiable with a bounded and a globally Lipschitz derivative). Then $C:=\sigma^{2} \in \mathcal{C}^{1,1}_{loc}(\mathbb{R}^d, {\mathbb{S}^d_{+}})$ and  
	\begin{equation*}
 \langle u, \sum_{j=1}^{d} D \sigma^j(x)\sigma^j(x) \rangle = \langle u, \sum_{j=1}^{d} D C^j(x)(CC^+)^j(x) \rangle, \quad \mbox{for all } x  \in \mathcal{D} \mbox{ and  } u \in \Ker\sigma(x).
	\end{equation*}
\end{proposition}

\begin{proof}  	Fix $x \in \mathcal{D}$ and $u \in \Ker \sigma(x)$.  	
	By using Definition \ref{defjacobianmatrix} and Proposition \ref{propdifferentiationrules} in the Appendix,  we first compute that   
	\begin{equation*}
	D C(x)=D (\sigma(x)^{2}) =  (\sigma(x)\otimes I_d  )D \sigma(x)+ (I_d \otimes \sigma(x))D \sigma(x),
	\end{equation*}
	which clearly shows that $C$ is $\mathcal{C}^{1,1}_{loc}$. It then follows from Proposition \ref{propappendixformulas} and  the fact that $ u \in \Ker\sigma(x)$    that  	
	\begin{align*}
	 (I_d \otimes u^\top)D C(x)C(x)C(x)^+ &=   (\sigma(x) \otimes u^\top)D \sigma(x) C(x)C(x)^+ .
	\end{align*}
Observe now that  $C(x)C(x)^+\sigma(x) = \sigma(x)$ since  $C(x)= \sigma(x)^{2}$ (use the spectral decomposition of $\sigma$ as in Proposition \ref{proppseudoinversespectral}). Using  Proposition \ref{propappendixformulas} again, the above implies that 
	\begin{align*}
	\Tr\left[(I_d \otimes u^\top)D C(x)C(x)C(x)^+\right]&= \Tr\left[\sigma(x)( {I_d}\otimes u^\top)D \sigma(x)C(x)C(x)^+ \right]\\
	&=\Tr\left[( {I_d}\otimes u^\top)D \sigma(x) \sigma(x)\right].
	\end{align*}
	Then, by   Proposition   \ref{propappendixformulas} and \ref{propdifferentiationrules},   
	\begin{eqnarray}
	 \langle u, \sum_{j=1}^{d} D \sigma^j(x)\sigma^j(x) \rangle &=& 
	 \sum_{j=1}^{d} u^\top  D (\sigma(x)e_j)\sigma(x)e_j \nonumber\\
	 &=& \sum_{j=1}^{d} u^\top   {(e_j^\top\otimes I_{d} )}D \sigma(x)\sigma(x)e_j \nonumber\\
	  &=& \sum_{j=1}^{d}  e_j^\top (I_d \otimes u^\top)D \sigma(x)\sigma(x)e_j \nonumber\\
	 &=& \Tr \left[ (I_d \otimes u^\top)D \sigma(x)\sigma(x)\right] \nonumber\\
	 &=& 	\Tr\left[(I_d \otimes u^\top)D C(x)C(x)C(x)^+\right] \nonumber\\
	 &=& \langle u, \sum_{j=1}^{d} D C^j(x)(CC^+)^j(x) \rangle,\label{eq: trace prod vec}
	\end{eqnarray}
	in which the last identity follows by reproducing exactly the same computations in the reverse order with $C$ in place of $\sigma$. 
\end{proof}

 The following provides another formulation of   \eqref{ourconddrift} that highlights the notion of projection on the image of $C$. 

\begin{remark}[Interpretation of the projection formulation]\label{rem : projection}
Fix $x \in \partial\mathcal{D}$ and assume that the spectral decomposition of $C$ at $x$ takes the form $C(x)=Q(x)\diag{\lambda_1(x),\ldots,\lambda_r(x),0,\ldots,0}Q(x)^\top$, where $Q(x)Q(x)^\top=I_d$ and $\lambda_j(x) >0$ for all $1 \leq j \leq r$. Hence, the $r$-first columns of $Q(x)$, denoted by $(q_1,\ldots,q_r)=(q_1(x),\ldots,q_r(x))$, span the image of $C(x)$ and the projection matrix on the image of $C(x)$ is given by $C(x)C(x)^+=\sum_{j=1}^{r} q_jq_j^\top$, see Propositions \ref{proppseudoinverseproj} and \ref{proppseudoinversespectral} in the Appendix {and recall that $q_{j}$ is a column vector}.  Thus,  by  \eqref{eq: trace prod vec} in the proof of Proposition \ref{propequivalencedaprato} and Proposition   \ref{propappendixformulas} in the Appendix, 
\begin{eqnarray*}
\langle u, \sum_{j=1}^{d} D C^j(x)(CC^+)^j(x) \rangle &=& 	\Tr\left[(I_d \otimes u^\top)D C(x)C(x)C(x)^+\right]\\
&=& \sum_{j=1}^{r} \Tr\left[(I_d \otimes u^\top)D C(x)q_jq_j^\top\right]  \\
&=& \sum_{j=1}^{r} \Tr\left[q_j^\top(I_d \otimes u^\top)D C(x)q_j\right] \\
&=& \sum_{j=1}^{r}  u^\top(q_j^\top\otimes I_d )D C(x)q_j
\end{eqnarray*}
so that, by Proposition \ref{propdifferentiationrules}, 
\begin{eqnarray*}
\langle u, \sum_{j=1}^{d} D C^j(x)(CC^+)^j(x) \rangle  &=& \langle u , \sum_{j=1}^{r} D(Cq_j)(x)q_{j} \rangle= \langle u , \sum_{j=1}^{r} D_{q_j}(Cq_j)(x) \rangle
\end{eqnarray*}
in which $D_{q_j}$ is the directional derivative with respect to $q_j$:
\begin{equation*}
D_{q_j}(Cq_j)(x):=\displaystyle\lim_{t \rightarrow 0} \frac{C(x+t q_j)q_j- C(x)q_j}{t}.
\end{equation*}
Therefore   \eqref{ourconddrift} reads $\langle u , b(x) - \frac{1}{2} \sum_{j=1}^{r} D_{q_j}(Cq_j)(x)\rangle \le 0$. Otherwise stated, $C$ is first projected onto the basis of the image of $C(x)$ before being derived only in the directions of $(q_1,...,q_r)$. This is clearly consistent with   \eqref{ourcondcov} {that states} that there {cannot} be any transverse diffusion of $C(x)$ to the boundary. Therefore, the drift $b(x)$ should only compensate the tangential diffusion given by the projection onto the image of $C(x)$ in order to keep the diffusion in the domain. 
\end{remark}

{Let us conclude this section with an additional comment for the jump-diffusion case.  }
{\begin{remark}[Adding jumps]
		Note that jumps could be included in the dynamics of $X$. Based on the current work, we provide in \cite{aj16} an extension of the first order characterization of Theorem \ref{MainTheorem} to the jump-diffusion case. We
		also derive an equivalent formulation in the semimartingale framework.
\end{remark}}

%%%%%%%%%%%%%%%%%%%%%%%%%%%%%%%%%%%%%%%%%
%%%%%%%%%%%%%%%%%%%%%%%%%%%%%%%%%%%%%%%%%
\section{Necessary {conditions}}\label{SectionNec}

In this section, we prove  that the {conditions} of Theorem \ref{MainTheorem} {are} necessary for $\Dc$ to be invariant. 

Our general strategy is {similar to \cite{buc}}. { We fix $x \in \Dc$ and} we consider a smooth function  $\phi:\mathbb{R}^d \mapsto \mathbb{R}$ such that {$\displaystyle\max_{ \mathcal{D}}\phi=\phi(x)$}. Since $\mathcal{D}$ is  {stochastically invariant}, {let $X$ be a $\Dc$-valued solution starting from $X_0=x$. In particular,} $\phi(X_t) \leq \phi(x)$,   for all $t \geq 0$. Then, {if $\sigma$ is sufficiently smooth}, by applying It\^{o}'s Lemma twice, we obtain
\begin{equation*}\label{eqdoublestochasticintro}
 			\int_{0}^{t} \mathcal{L}\phi(X_s)ds + \int_{0}^{t} \left(  D\phi \sigma(x) + \int_0^s \mathcal{L}(D\phi\sigma)(X_r) dr + \int_0^s D ( D\phi{\sigma})\sigma(X_r) dW_r\right)^\top dW_s \leq 0.
 		\end{equation*}
		{Recall Remark \ref{rmkchoiceofsigma} for the definition of the infinitesimal generator $\Lc$.}
Given    (now standard) estimates on the small time behavior of single and double stochastic integrals, see e.g.~\cite{buc,cst05a}, this readily implies 
		\begin{equation*} 
 		D\phi(x)\sigma(x)=0\quad \mbox{ and } \quad \langle D\phi(x), b(x) - \frac{1}{2}\sum_{j=1}^{d} D  \sigma^j(x)\sigma^j(x) \rangle \leq 0,
 		\end{equation*}
		under appropriate regularity conditions. 
		It remains to choose a suitable test function $\phi$, i.e.~such that $ D\phi(x)=u{^\top}$,  to deduce that 
 \eqref{ourcondcov}-\eqref{ourconddrift}	must hold {when $\sigma$ is differentiable}, recall Proposition \ref{propequivalencedaprato}. 
 
 In our setting, one can however not differentiate $\sigma^{j}$ in general. To surround this problem the above can be rewritten in term of the covariance matrix $C$. The projection term in  \eqref{ourcondcov}-\eqref{ourconddrift} will appear through a conditioning argument. 
 
 In order to separate the difficulties, we shall first consider the case where {$C$ admits} a locally smooth spectral decomposition. The general case will be handled in Section \ref{Sec: change state space} below.
	
%%%%%%%%%%%%%%%%%%%%
\subsection{The case of distinct eigenvalues}

As mentioned above, we shall first make profit of  having distinct eigenvalues before considering the general case. The main idea consists in using the spectral decomposition of $C$ in the form $Q\Lambda Q^{\top}$ in which $Q$ is an orthogonal matrix and $\Lambda$ is diagonal positive semi-definite. Then,  the dynamics of $X$ can be written as 
$$
dX_{t}=b(X_{t})dt+Q(X_{t}) {\Lambda(X_{t})}^{\frac12}dB_{t}
$$
in which $B=\int_{0}^{\cdot} Q(X_{s})^{\top}dW_{s}$ is a Brownian motion. If $Q$ and $\Lambda$ are smooth enough, then we can apply the same ideas as the one exposed at the beginning of this section. An additional locali{z}ation and conditioning argument will {allow} us to reduce to the case where $\Lambda$ has only (strictly) positive entries. 

{Note that eigenvalues and the eigenvectors can always be  chosen   measurable. However, multiple eigenvalues and their corresponding eigenvectors can fail to have the same regularity as $C$. To ensure a sufficient regularity, we therefore assume in the following Lemma that non-zero eigenvalues are distinct. The general case will be treated later, thanks to a change of variable argument, see Section \ref{Sec: change state space} below.} 

\begin{lemma}\label{lemmadistincteigenmatrix}
	Assume that $C \in \mathcal{C}^{1,1}_{loc}(\mathbb{R}^d,\mathbb{S}^d)$.
	Let $x \in \mathcal{D}$ be such that the spectral decomposition of $C(x)$ is given by
	\begin{equation}\label{eq: sectral decompo}
	C(x)=Q(x)\diag{\lambda_1(x),\ldots ,\lambda_r(x),0,\ldots ,0}Q(x)^\top
	\end{equation}
	with $\lambda_1(x) >\lambda_2(x) >\cdots >\lambda_r(x)>0$ and $Q(x)Q(x)^\top=I_d$, $r\le d$.
	
	Then there {exist} an open (bounded) neighborhood $N(x)$ of $x$ and two measurable $\M^{d}$-valued functions on $\bbR^d$, $y \mapsto Q(y):=[q_1(y)\cdots q_d(y)]$ and $y \mapsto \Lambda(y):=\diag{\lambda_1(y),\ldots ,\lambda_d(y)}$ such that
	\begin{enumerate}[{\rm (i)}]
		\item
		$C(y)=Q(y)\Lambda(y)Q(y)^\top$ and $Q(y)Q(y)^\top=I_d$, for all $y \in \bbR^d$,
		\item
		 $\lambda_1(y) >\lambda_2(y) >...>\lambda_r(y)>\max\{\lambda_{i}(x), r+1\le i\le d\}\vee 0$, for all $y \in N(x)$,
		\item
		$\bar{\sigma}: y \mapsto \bar{Q}(y)   \bar{\Lambda}(y)^{\frac12}$ is $C^{1,1}(N(x),\mathbb{M}^d)$, in which 
		  $\bar{Q}:=[q_1\cdots q_r\; 0\cdots 0]$ and  $\bar{\Lambda}=$ ${\rm diag}[\lambda_1,...,$ $\lambda_r,0,...,0]$.
		
	\end{enumerate}
	
	Moreover, we have: 
	\begin{equation}\label{eqsigmabarc}
	\langle u, \sum_{j=1}^{d} D  \bar{\sigma}^j(x) \bar{\sigma}^j(x) \rangle = \langle u, \sum_{j=1}^{d} D  C^j(x) (CC^+)^j(x)  \rangle, \quad \mbox{for all } u \in \Ker(C(x)).
	\end{equation}
\end{lemma}

\begin{proof} Note that the fact that $(q_{i})_{i\le d}$ can be chosen measurable is guaranteed when $(C,\Lambda)$ is measurable by the fact that each eigenvector solves a quadratic minimization problem, see e.g.~\cite[Proposition 7.33(p.153)]{ber78}. Moreover,  the continuity of the eigenvalues follows from {Weyl's perturbation theorem}, \cite[Corollary III.2.6]{bha}, and   the  smoothness of $(\bar \Lambda,\bar Q)$ is a consequence of \cite[Theorem 1]{mag85} since all the positive eigenvalues are simple and $C$ is $\mathcal{C}^{1,1}_{loc}(\mathbb{R}^d,\mathbb{S}^d)$. 
	
Let us now observe that any $u \in \Ker(C(x))$ satisfies	
	\begin{equation*}
	u^\top \bar{Q}(x)=u^\top \bar{\sigma}(x)=0.
	\end{equation*}
	Since  $\bar{C}:=\bar{\sigma}\bar{\sigma}^\top$  is  differentiable at $x$, the product rule of Proposition \ref{propdifferentiationrules} combined with Proposition \ref{propappendixformulas} yields
	\begin{align*}
	(I_d \otimes u^\top) D \bar{C}(x)&=(I_d \otimes u^\top)\left[ (\bar{\sigma}(x) \otimes I_d)D \bar{\sigma}(x) + (I_d \otimes \bar{\sigma}(x))D \bar{\sigma}(x)^\top\right]\nonumber
	\\
	&=(\bar{\sigma}(x) \otimes u^{\top})D \bar{\sigma}(x) 
	\\
	&=\bar{\sigma}(x)(I_{d} \otimes u^{\top})D \bar{\sigma}(x).
	\end{align*}

	Observing that $\bar{C}=\bar{\sigma}\bar{\sigma}^\top=C \bar{Q}\bar{Q}^\top$ and that  $\bar{Q}(x)\bar{Q}(x)^\top = C(x)C(x)^+$, we   get by similar computations:
	\begin{align*}
	(I_d \otimes u^\top) D\bar{C}(x)&= (I_d \otimes u^\top) \left[(C(x)C(x)^+ \otimes I_d)D C(x) + (I_d \otimes C(x))D \left(\bar{Q}\bar{Q}^\top \right)(x) \right]
		\\
			&=C(x)C(x)^+( I_d\otimes u^{\top})D C(x). 
	\end{align*} 

%Observing from Proposition \ref{proppseudoinversespectral} that $\bar{C}(x)=\bar{\sigma}(x)\bar{\sigma}(x)^\top=C(x)  (CC^+)(x)$, 
%we also get in a similar way \blue{plutôt écrire $\bar{C}(y)=\bar{\sigma}(y)\bar{\sigma}(y)^\top=C(y)C(x)C(x)^+$, pour bien montrer qu'on ne dérive pas $CC^+$ et $x$ est fixé!}
		%\begin{equation*}
		%(I_d \otimes u^\top)D \bar{C}(x)= (C(x)C(x)^+ \otimes u^{\top})D C(x)=C(x)C(x)^+(I_{d} \otimes u^{\top})D C(x).
		%\end{equation*} 
	Combining the above leads to
	\begin{equation*}
	\Tr\left[ (I_d \otimes u^\top) D \bar{\sigma}(x)\bar{\sigma}(x)\right] =\Tr\left[ (I_d \otimes u^\top) D C(x)C(x)C(x)^+\right],
	\end{equation*} 
	 which proves \eqref{eqsigmabarc} by similar computations as in the proof of \eqref{eq: trace prod vec}. 
\end{proof}

We can now adapt the arguments of \cite{buc}. {In the following we use the notion of  proximal normals. A vector $u \in \mathbb{R}^d$ is said to be a proximal normal to $\Dc$ at a point $x$ if  $ \|u\|=d_{\calD}(x+u)$, where  $d_\calD$ is the distance function to $\mathcal{D}$. We denote by 
	$\mathcal{N}^{1,prox}_{\mathcal{D}}(x)$ the cone spanned by all proximal normals.}
Note however that 
 \eqref{ourcondcov}-\eqref{ourconddrift}  holds at $x$ for all    proximal normals $u \in \mathcal{N}^{1,prox}_{\mathcal{D}}(x)$ if and only if it  holds for all     $u \in \mathcal{N}^{1}_{\mathcal{D}}(x)$.
 {Indeed,     
 	\begin{align}\label{eq:normallimsup}
 	 \mathcal{N}^{1,prox}_{\mathcal{D}}(x)\subset \mathcal{N}^1_{\mathcal{D}}(x)\subset \bar{\mbox{{co}}}\left( \displaystyle \limsup_{ \Dc \ni y \to x}  \mathcal{N}^{1,prox}_{\mathcal{D}}(y) \right),
 	 \end{align}
 	 where  $\limsup$ stands for the Painlevé-Kuratowski upper limit (see e.g.~\cite{aubf,da}) and $\bar{\mbox{co}}$ is the closed convex hull (see also \cite[Remark 4.2 (a)]{da}).}

\begin{lemma}\label{lemmadistincteigen}
	Assume that $\mathcal{D}$ is \textit{stochastically invariant} with respect to the diffusion \eqref{diffusionsdeinvariance}. Let $x\in \Dc$ and $C$ be as in Lemma \ref{lemmadistincteigenmatrix}. 
	Then, \eqref{ourcondcov} and \eqref{ourconddrift} hold at $x$ for   all $u \in \mathcal{N}^{1}_{\mathcal{D}}(x)$. 
\end{lemma}

\begin{proof} It follows from the discussion before our lemma that it suffices to prove our claim for $u \in \mathcal{N}^{1,prox}_{\mathcal{D}}(x)$.    
	Let $(X,W)$ denote a  {weak solution} starting at $X_0=x$ such that  $X_t \in \mathcal{D}$ for all $t \geq 0 $. If $x\notin \partial \Dc$, then  $\mathcal{N}^{1,prox}_{\mathcal{D}}(x)=\{0\}$ and there is nothing to prove. We therefore assume from now on that $x\in {\partial\Dc}$. We fix $u \in \mathcal{N}^{1,prox}_{\mathcal{D}}(x)$.
	
{Step 1.} We first claim that    there exists a function $\phi \in \mathcal{C}^{\infty}_b(\mathbb{R}^{d} ,\mathbb{R})$  with compact support in $N(x)$ such that $\displaystyle\max_{\mathcal{D}} \phi = \phi(x) = 0$ and $D \phi(x)= u{^\top}$. Indeed, it follows from \cite[Chapter 6.E]{roc} that one can find $\kappa >0$ such that
	$
	\langle u, y-x \rangle \leq \frac{\kappa}{2} \|y-x\|^2 $ for all $ y \in \mathcal{D}$.
	Then, one can set $
	\psi:=\langle u, \cdot-x \rangle - \frac{\kappa}{2} \|\cdot-x\|^{2}
	$
	and define $\phi:=\psi \rho$ in which $\rho$ is a $\mathcal{C}^{\infty}_b$ function with   values in $[0,1]$, compact support  included in $ N(x)$, and  satisfying $\rho=1$ in  a neighborhood of $x$.

{Step 2.}	 Since $\mathcal{D}$ is invariant under the diffusion $X$, $\phi(X_t) \leq \phi(x)$, for all $t\geq 0$. {From now on, we use the notations of Lemma \ref{lemmadistincteigenmatrix}. By the above and It\^{o}'s lemma:} 
	\begin{align*}
	0\ge & \int_{0}^{t } \mathcal{L}\phi(X_{s})ds + \int_{0}^{t} D\phi(X_{s}) \sigma(X_{s})dW_s   
	=\int_{0}^{t } \mathcal{L}\phi(X_{s})ds + \int_{0}^{t} (D\phi Q \Lambda^{\frac12}Q^\top)(X_{s})dW_s 
	\end{align*}
	in which    $\Lc$ is the infinitesimal generator of $X$. Let us define the Brownian motion $B=\int_{0}^{\cdot} Q(X_{s})^\top dW_s$, recall that $Q$ is orthogonal, together with $\bar{B}=\Lambda(x)\Lambda(x)^+ B =(B^1,..,B^r,0,...,0)^\top $ and $\bar{B}^{\perp}=(I_d-\Lambda(x)\Lambda(x)^+ )B=(0,...,0,B^{r+1},...,B^d)$, recall Proposition \ref{proppseudoinversespectral}. Since $Q {\bar{\Lambda}}^{\frac12}=\bar{Q} {\bar{\Lambda}}^{\frac12}$, the above inequality can be written in the form 
	\begin{align*}
	0\ge & \int_{0}^{t } \mathcal{L}\phi(X_{s})ds + \int_{0}^{t} D\phi(X_{s}) \bar \sigma(X_{s})d\bar B_s + \int_{0}^{t} (D\phi Q \Lambda^{\frac12})(X_{s})d\bar B^{\perp}_s.
	\end{align*}
	Let $(\mathcal{F}_s^{\bar{B}})_{s\ge 0}$ be the completed filtration   generated by $\bar B$. By  \cite[Corollaries 2 and 3 of Theorem 5.13]{lip01}, \cite[Lemma 14.2]{kur},  and the fact that the martingale $\bar B^{\perp}$ is independent of $\bar B$, we obtain 
	\begin{align*}
	0&\ge  \int_{0}^{t } \mathbb{E}_{\mathcal{F}_s^{\bar{B}}} [\mathcal{L}\phi(X_{s})]ds + \int_{0}^{t} \mathbb{E}_{\mathcal{F}_s^{\bar{B}}} [D \phi(X_{s}) \bar \sigma(X_{s})]d\bar B_s\\
		&= \int_{0}^{t } \mathbb{E}_{\mathcal{F}_s^{\bar{B}}} [\mathcal{L}\phi(X_{s})]ds + \int_{0}^{t} \mathbb{E}_{\mathcal{F}_s^{\bar{B}}} [D \phi(X_{s}) \bar \sigma(X_{s})]d B_s,
	\end{align*}	
	{where the last equality holds because  the $(d-r)$ columns of $\bar \sigma$ are 0.} We now apply Lemma \ref{lemmakomlos} below to $(D  \phi \bar{\sigma})(X)$ and use  \cite[Corollaries 2 and 3 of Theorem 5.13]{lip01} and \cite[Lemma 14.2]{kur} again to find a bounded adapted process $\eta$ such that 	
	\begin{equation}\label{eq: for applying small time behavior}
	0\ge \int_{0}^{t} \theta_sds + \int_{0}^{t} \left( \alpha + \int_{0}^{s} \beta_rdr +  \int_{0}^{s} \gamma_r dB_r \right)^\top  dB_s  
	\end{equation}
	where 
	\begin{eqnarray*}
		\theta:= \mathbb{E}_{\mathcal{F}_{\cdot}^{\bar{B}}}\left[\mathcal{L}\phi(X_{\cdot})\right]&,&
	\alpha^{\top}:=  ( D\phi \bar{\sigma})(x) = u^\top  Q(x){\Lambda(x)}^{\frac12} \\
	\beta := \mathbb{E}_{\mathcal{F}_\cdot^{\bar{B}}}\left[D (D \phi \bar{\sigma} )(X_{\cdot})b(X_{\cdot}) + \eta_{\cdot} \right]&,&
	\gamma :=  \mathbb{E}_{\mathcal{F}_\cdot^{\bar{B}}}\left[D (D \phi  \bar{\sigma})\bar{\sigma}(X_{\cdot})\right],  
	\end{eqnarray*} 
	recall from Step 1 that $D\phi(x)=u^{\top}$.

{Step 3.} We now check that we can apply   Lemma \ref{lemmadoubleintsto} below.  First note that all the above processes are bounded. This follows from Lemma \ref{lemmadistincteigenmatrix}, \eqref{growthconditions} and the fact that $\phi$ has compact support. 
In addition, given $T >0$, the independence of the increments of  $\bar{B}$ implies that $\theta_s= \bbE_{\calF^{\bar{B}}_T}\left[ \calL \phi (X_s) \right]$ for all $s \leq T$. It follows that $\theta$ is a.s.~continuous at $0$. 

Similarly, $\gamma =  \mathbb{E}_{\mathcal{F}_T^{\bar{B}}}\left[D (D \phi  \bar{\sigma})\bar{\sigma}(X_{\cdot})\right]$ on $[0,T]$.
Moreover, since $ D\phi\bar{\sigma}$ is $\mathcal{C}^{1,1}$, $F:=D ( D\phi  \bar{\sigma})\bar{\sigma}$ is Lipschitz and  Jensen's inequality combined with \eqref{holdertrajectories} implies that we can find $L'>0$ such that  
	\begin{equation*}
	\bbE\left[ \|\gamma_s - \gamma_r \|^{4} \right] \leq \bbE\left[ \|F(X_s) - F(X_r) \|^{4} \right]  \leq L' |s-r|^{2}, \quad \mbox{for all } 0 \leq s,r \leq 1.
	\end{equation*}
	
 By  Kolmogorov's continuity criterion, up to considering a suitable modification,   $\gamma$  has $\epsilon$-Hölder sample paths for all $0<\epsilon<\frac14$, in particular $\int_{0}^{t} \|\gamma_s - \gamma_0\|^2ds = O(t^{1+\epsilon})$ for  $0<\epsilon<\frac12$.

{Step 4.} In view of Step 3, we can apply  Lemma \ref{lemmadoubleintsto} to \eqref{eq: for applying small time behavior} to deduce  that $\alpha=0$ and 
		$\theta_0-\frac{1}{2}\Tr(\gamma_0) \leq 0$. {Multiplying the first equation  by $ {\Lambda(x)}^{\frac12} Q^{\top}(x)$ implies that $0=\alpha^{\top} {\Lambda(x)}^{\frac12} Q^{\top}(x)= u^\top Q(x) {\Lambda(x)}^{\frac12}  {\Lambda(x)}^{\frac12} Q^{\top}(x)   = u^\top  C(x)$,} or equivalently $C(x)u=0$ since $C(x)$ is symmetric. The second identity combined with $D\phi(x)=u^{\top}$ and Proposition \ref{propdifferentiationrules} shows that  
	\begin{align*}
	0\ge&  \Lc\phi(x)-\frac12 \Tr\left[ {  \bar \sigma^{\top} D^{2}  \phi \bar \sigma} +( {I_d\otimes u^\top})D \bar{\sigma}\bar{\sigma}\right](x) 
	 = u^{\top}b(x)-\frac12 \Tr\left[( {I_d\otimes u^\top}) D \bar{\sigma}\bar{\sigma}\right](x),
	\end{align*}
which is equivalent to \eqref{ourconddrift} by  \eqref{eqsigmabarc}  and  similar computations as in the proof of \eqref{eq: trace prod vec}. 		 
\end{proof}

The rest of this section is dedicated to the proof of the two technical lemmas that were used above. Our first result is a slight extension of It\^{o}'s lemma to only $\mathcal{C}^{1,1}$ function. It is based on a simple application of   Koml\'os lemma (note that the assumption that $f$ has a compact support in the following is just for convenience, it can obviously be removed by a localization argument, {in which case the process $\eta$ is only locally bounded}). 

\begin{lemma}\label{lemmakomlos}
	Assume that $b$ and $\sigma$ are continuous and that there exists a solution $(X,W)$  to \eqref{diffusionsdeinvariance}. Let $f \in \mathcal{C}^{1,1} (\mathbb{R}^d,\mathbb{R})$ have compact support. Then, there exists an adapted bounded process $\eta$ such that 	
	\begin{equation*}\label{eqkomlos}
	f(X_t)= f(x) + \int_{0}^{t} \left( D  f(X_s)  b(X_s) + \eta_s \right) ds + \int_{0}^{t} D  f(X_s) \sigma(X_s)dW_s
	\end{equation*}
	for all $t\ge 0$.
\end{lemma}

\begin{proof} Since $f\in \calC^{1,1}$ has a compact support, 	we can find a sequence $(f_n)_{n}$ in $\mathcal{C}^\infty$ with  compact support (uniformly) and a constant $K>0$ such that 
	\begin{enumerate}[(i)]
		\item 
		$\|D^2 f_n\| \leq K$, 
		\item
		$\|f_n-f\|+\|D  f_n-D  f\| \leq \frac{K}{n}$,	
		\end{enumerate}
	for all $n\ge 1$. This is obtained by considering a simple mollification of $f$. By applying Itô's Lemma to $f_n(X)$, we get 
	\begin{equation*}
	f_n(X_t)= f_n(x) + \int_{0}^{t}   D  f_n(X_s)  b(X_s) ds +  \int_{0}^{t} \eta^{n}_{s}ds + \int_{0}^{t} D  f_n(X_s)  \sigma(X_s)dW_s
	\end{equation*}
	in which $\eta^{n}:=\frac{1}{2}{\Tr[D ^2 f_n\sigma\sigma^{\top}](X)}$.  Since {$\sigma\sigma^{\top}$} is continuous, (i) above implies that $(\eta^{n})_{n}$ is uniformly bounded in $L^{\infty}(dt \times d\mathbb{P})$.  By \cite[Theorem 1.3]{del99}, there exists $(\widetilde{\eta}^n) \in \Conv( \eta^k, k\geq n )$  such that $\widetilde{\eta}^n \rightarrow \eta$ $dt\otimes d\mathbb{P}$ almost surely. 
	Let $N_n \geq 0$ and $(\lambda^n_k)_{n \leq k \leq N_n}\subset [0,1]$ be such that 
	$\widetilde{\eta}^n=\sum_{k=n}^{N_n} \lambda^n_k \eta^k$ and $\sum_{k=n}^{N_n} \lambda^n_k =1$. Set  $\widetilde{f}_n:=\sum_{k=n}^{N_n} \lambda^n_k f_k$. Then, 	
	\begin{equation}\label{eqitofntilde}
	\widetilde{f}_n(X_t)= \widetilde{f}_n(x) + \int_{0}^{t} D  \widetilde{f}_n(X_s)  b(X_s) ds +  \int_{0}^{t} \widetilde{\eta}^n_s  ds + \int_{0}^{t} D  \widetilde{f}_n(X_s) \sigma(X_s)dW_s.
	\end{equation}
	By   dominated convergence, 
	$
	\int_{0}^{t} \widetilde{\eta}^n_s  ds$ converges a.s.~to $\int_{0}^{t} \eta_s  ds
	$.  
	Moreover,  (ii) implies that  
	\begin{equation*}
	\|\widetilde{f}_n(X_t)-f(X_t)\| \leq  \sum_{k=n}^{N_n} \lambda^n_k \|\widetilde{f}_k(X_t)-f(X_t)\| \leq \sum_{k=n}^{N_n} \lambda^n_k \frac{K}{k}\leq \frac{K}{n},
	\end{equation*} 
	so that   $\widetilde{f}_n(X_t)$ converges a.s.~to $f(X_t)$. Similarly, 
	$$ \int_{0}^{t} D\widetilde{f}_n(X_s) \sigma(X_s)dW_s \rightarrow \int_{0}^{t} D f(X_s) \sigma(X_s)dW_s\mbox{ and } \int_{0}^{t} D  \widetilde{f}_n(X_s)  b(X_s) ds \rightarrow \int_{0}^{t} {D  f}(X_s)  b(X_s) ds 
	$$ 
	in 	$L^2(\Omega,\mathcal{F},\mathbb{P})$ as $n \rightarrow \infty$, and therefore a.s.~after possibly considering a subsequence. It thus remains to send $n\to \infty$ in \eqref{eqitofntilde} to obtain the required result.  
\end{proof}

The following adapts \cite[Lemma 2.1]{buc} to our setting, see also \cite{brud,cst05a,cst05b}. 

\begin{lemma}\label{lemmadoubleintsto}
Let $(W_t)_{t \geq 0}$ denote a standard $d$-dimensional Brownian motion on a filtered probability space $(\Omega,\mathcal{F},(\mathcal{F}_t)_{t \geq 0},\mathbb{P})$. Let $\alpha \in \mathbb{R}^d$ and $(\beta_t)_{t\geq 0}$,  $(\gamma_t)_{t\geq 0}$ and  $(\theta_t)_{t\geq 0}$ be  adapted processes taking values respectively in $\mathbb{R}^d$, $\mathbb{M}^{d }$ and $\mathbb{R}$ and satisfying  
	
	\begin{enumerate}[{\rm (1)}]
		\item
		$\beta$ is bounded,
		\item  $\int_{0}^t \|\gamma_s\|^2 ds < \infty$, for all $ t\geq 0$,
		\item
		there exists a random variable $\eta >0$ such that a.s.
		\begin{equation}\label{eqgammacontinuity}
		\int_{0}^{t}\|\gamma_s - \gamma_0\|^2ds = O(t^{1+\eta}) \quad \mbox{for } t \rightarrow 0,
		\end{equation}
		\item
		$\theta$ is a.s.~continuous at $0$.
	\end{enumerate}
	Suppose that for all $t \geq 0$
	\begin{equation}\label{eqintegraledouble}
	\int_{0}^{t} \theta_s ds + \int_{0}^{t} \left( \alpha + \int_{0}^{s} \beta_rdr +  \int_{0}^{s} \gamma_r dW_r \right)^\top  dW_s  \leq 0.
	\end{equation}
	
	Then, 
	\begin{enumerate}[{\rm (a)}]
		\item
		$\alpha=0$,
		\item 
		$-\gamma_0 \in \mathbb{S}^d_+$,
		\item
		$\theta_0-\frac{1}{2}\Tr(\gamma_0) \leq 0$. 
	\end{enumerate}
	
\end{lemma}

\begin{proof}	Since $(W^i_t)^2 = 2 \int_0^t W^i_s d W^i_s + t$,   \eqref{eqintegraledouble} reduces to 
	\begin{equation*}
	(\theta_0- \frac{1}{2} \Tr(\gamma_0)) t + \sum_{i=1}^d \alpha^i W^i_t + \sum_{i=1}^d \frac{\gamma^{ii}_0}{2} (W^i_t)^2 + \sum_{1 \leq i \neq j \leq d} \gamma^{ij}_0 \int_0^t W^i_s d W^j_s + R_t \leq 0,
	\end{equation*}
where
	\begin{eqnarray*}
	R_t &=& \int_{0}^{t} (\theta_s - \theta_0) ds + \int_{0}^{t} \left(  \int_{0}^{s} \beta_rdr \right)^\top  dW_s  +  \int_{0}^{t}\left(\int_{0}^{s} (\gamma_r -\gamma_0) dW_r \right)^\top  dW_s \\
	 &=:&  R^1_t + R^2_t + R^3_t.
	\end{eqnarray*}
	In view of    \cite[Lemma 2.1]{buc}, it suffices to show that  $R_t/t \rightarrow 0$ almost surely. To see this, first note that $R^1_t = o(t)$  a.s.~since $\theta$ is continuous at $0$. Moreover, \cite[Proposition 3.9]{cst05a} implies that $R^2_t = o(t)$ a.s., as $\beta$ is bounded.  
	
	It remains to prove that $R^3_t = o(t)$ a.s. To see this, define $M^{ij}=\gamma^{ij} - \gamma^{ij}_0$
	and $M^i=\int_0^\cdot \sum_{j=1}^{d}M^{ij}_r dW^j_r$ for all $1\leq i,j\leq d$. The continuity assumption \eqref{eqgammacontinuity} implies that  $\langle M^i \rangle_s = O(s^{1+\eta})$ almost surely. By the Dambis-Dubins-Schwarz theorem, $(M^i_s)_{s\geq 0}$ is therefore a time-changed Brownian motion, see e.g.~\cite[Theorem V.1.6]{rev}. By the law of iterated logarithm for Brownian motion $(M^i_s)^2 = O(s^{1+\frac{\eta}{2}})$ almost surely. Hence, $\langle R^3\rangle_t= O(t^{2+\frac{\eta}{2}})$ almost surely. By applying the Dambis-Dubin-Schwarz theorem and the law of iterated logarithm again, we obtain that   $R^3_t = o(t)$ a.s.
\end{proof}

%%%%%%%%%%%%%%%%%%%%
\subsection{The general case}\label{Sec: change state space}

We can now turn to the general case. 

\begin{proposition}[Necessary conditions of Theorem \ref{MainTheorem}]\label{propnecessity}
	Let the conditions of Theorem \ref{MainTheorem}  hold and assume that $\calD$  is \textit{stochastically invariant} with respect to the diffusion \eqref{diffusionsdeinvariance}. Then conditions \eqref{ourcondcov} and \eqref{ourconddrift} hold for all $x \in \calD$ and  $u \in \mathcal{N}^{1}_{\mathcal{D}}(x)$. 
\end{proposition}

\begin{proof}  If $x$ lies in the interior of $\mathcal{D}$, then $\mathcal{N}^{1}_{\mathcal{D}}(x)=\{0\}$ and there is nothing to prove. 
We therefore assume from now on that   $x \in \partial\mathcal{D}$. Let 	$\Lambda$ and $Q$ be defined through the spectral decomposition of $C$, as in \eqref{eq: sectral decompo} but with only  $\lambda_{1}(x)\ge \cdots\ge \lambda_{d}(x)$. We shall perform a change of   variable to reduce to the conditions of Lemma \ref{lemmadistincteigen}. To do this, we fix $0 < \epsilon < 1$ and define 
\begin{equation*}
 A^{\epsilon} = Q(x) \diag{\sqrt{(1-\epsilon)},\sqrt{(1-\epsilon)^2},\ldots,\sqrt{(1-\epsilon)^d}} Q(x)^\top. 
\end{equation*}

Since $\mathcal{D}$ is invariant with respect to the diffusion $X$, $ \mathcal{D}^{\epsilon}:=A^{\epsilon} \mathcal{D}$ is invariant with respect to the diffusion $X^{\epsilon}:= A^{\epsilon}X$.  Note that 
$$
dX^{\epsilon}=b_{\epsilon}(X^{\epsilon})dt+C_{\eps}(X^{\epsilon})^{\frac12}dW
$$
in which
\begin{eqnarray*}
b_{\epsilon}  :=  A^{\epsilon} b( (A^{\epsilon})^{-1} \cdot) &\mbox{ and }&
C_{\epsilon} := A^{\epsilon}  C( (A^{\epsilon})^{-1} \cdot)   ( A^{\epsilon})^\top
\end{eqnarray*}
have the same regularity and growth as $b$ and $C$. Moreover,   the positive eigenvalues of $C_{\epsilon}$ are all distinct at $x^{\epsilon}:=A^{\epsilon} x$, as $C_{\epsilon}(x^{\eps})=Q(x) \diag{{(1-\epsilon)}\lambda_{1}(x),\ldots,{(1-\epsilon)^d}\lambda_{d}(x)} Q(x)^\top$. We can therefore apply Lemma \ref{lemmadistincteigen} to $(X^{\epsilon},\mathcal{D}^{\epsilon})$:  
	\begin{subnumcases}{}
C_{\epsilon}(x^{\epsilon})u_{\epsilon}=0 \label{eqcase(ii)cov}\\
	\langle  u_{\epsilon}, b_{\epsilon}(x^{\epsilon})-\frac{1}{2} \sum_{j=1}^{d} D C_{\epsilon}^j(x^{\epsilon})(C_{\epsilon}C_{\epsilon}^+)^j(x^{\epsilon})    \rangle \leq 0 \label{eqcase(ii)drift}
	\end{subnumcases}
for all  $u_{\epsilon} \in \mathcal{N}^{1}_{A^{\epsilon} \mathcal{D}}(x^{\epsilon})$. 
We  now easily verify that $ \mathcal{N}^{1}_{A^{\epsilon}\mathcal{D}} (x^{\epsilon}) = (A^{\epsilon})^{-1} \mathcal{N}^{1}_{\mathcal{D}}(x)$, recall the definition in \eqref{eq: def first order cone}. Finally, by sending $\epsilon \rightarrow 0$ in  \eqref{eqcase(ii)cov} and \eqref{eqcase(ii)drift}, we get by continuity: 
	\begin{subnumcases}{}
C(x)u=0 \nonumber\\ 
	\langle  u, b(x)-\frac{1}{2} \sum_{j=1}^{d} D C^j(x)(CC^+)^j(x)    \rangle \leq 0,\nonumber
	\end{subnumcases}
	for all $u\in \mathcal{N}^{1}_{\mathcal{D}}(x)$, 
which ends the proof.
\end{proof}

%%%%%%%%%%%%%%%%%%%%%%%%%%%%

\section{Sufficient conditions}\label{SectionSuf}

In this section, we prove that the necessary conditions of Proposition \ref{propnecessity} are also sufficient. We start by showing in Proposition \ref{propmaximum} that   \eqref{ourcondcov} and \eqref{ourconddrift}  imply that the generator $\mathcal{L}$ of $X$ satisfies the \textit{positive maximum principle}: $\mathcal{L}\phi(x) \leq 0$ for any $x \in \mathcal{D}$ and any function $\phi \in \mathcal{C}^2(\mathbb{R}^d,\mathbb{R})$  such that $\displaystyle\max_{ \mathcal{D}} \phi=\phi(x) \geq 0$, see e.g.~\cite[p165]{eth}. Then, classical arguments, mainly \cite[Theorem 4.5.4]{eth}, yield the existence of a solution to the corresponding martingale problem  that stays in $\mathcal{D}$, see  Proposition \ref{propsufficiency} below. 
\vs2

The following proposition is inspired by \cite[Remark 5.6]{da}. 

\begin{proposition}\label{propmaximum}
	Under the assumptions of Theorem \ref{MainTheorem}, assume that   \eqref{ourcondcov}-\eqref{ourconddrift} hold for all $x \in \calD$ and  $u \in {\mathcal{N}^{1}_{\mathcal{D}}(x)}$. Then, the generator $\mathcal{L}$  satisfies the \textit{positive maximum principle}. 
\end{proposition}

\begin{proof} 
We fix $x \in \mathcal{D}$. For $1 \leq j \leq d$, let us consider the following deterministic control system:
\begin{equation}\label{detcontrolsystem2}\begin{cases}
y'(t)=C(y(t))  \sigma(x)^+ e_j \\
y(0)=x,
\end{cases}
\end{equation}
where $\sigma(x)^{+}$ is the  pseudoinverse  of $\sigma(x)$. Since $C$ is locally Lipschitz and verifies condition (\ref{ourcondcov}), \cite[Proposition 2.5]{da} {combined with \eqref{eq:normallimsup}}  implies that  $\mathcal{D}$ is  invariant with respect to the deterministic control system \eqref{detcontrolsystem2}. Then, by definition of the second order normal cone in \eqref{eq: def second order cone}, 
\begin{equation*}
\langle u , y(\sqrt{h})-x \rangle + \frac{1}{2}  \langle v(y(\sqrt{h})-x) , y(\sqrt{h})-x \rangle \leq o(||y(\sqrt{h})-x||^2)
\end{equation*}
 for any $(u,v) \in  \mathcal{N}^2_{\mathcal{D}}(x) $.
On the other hand, since $C$ is $\mathcal{C}^{1,1}_{loc}$, a Taylor expansion around $0$ yields
\begin{equation*}
y(\sqrt{h}) = x +\sqrt{h} C(x) \sigma(x)^+ e_j  + \frac{h}{2} (e_j^\top \sigma(x)^+ \otimes I_d)D   C(x)C(x)  \sigma(x)^+ e_j  + o(h),
\end{equation*}
recall Proposition \ref{propdifferentiationrules} and note that $(\sigma^{+})^{\top}=\sigma^{+}$ since $\sigma$ is symmetric. 
Now observe that $u\in \mathcal{N}^{1}_{\mathcal{D}}(x)$  whenever $(u,v)\in   \mathcal{N}^2_{\mathcal{D}}(x) $. In particular,  $u^\top C(x)=0$ under \eqref{ourcondcov}. Combining the above, and recalling Proposition \ref{propappendixformulas} then leads to 
\begin{equation*}\label{tempinegaliteh}
\frac{h}{2} e_j^\top(\sigma(x)^+ \otimes u^\top)D   C(x)C(x)  \sigma(x)^+ e_j  + \frac{h}{2} e_j^\top \sigma(x)^+  C(x) v C(x) \sigma(x)^+ e_j \leq o(h).
\end{equation*}
Note that $\sigma^{+}\sigma^{+}=C^{+}$ and that $C\sigma^{+}\sigma^{+}C=CC^{+}C=C$, see e.g.~Definition \ref{defpseudoinverse} and Proposition \ref{proppseudoinversespectral}, {and recall that $(\sigma(x)^+ \otimes u^\top)=\sigma(x)^+(I_{d} \otimes u^\top)$ by Proposition \ref{propappendixformulas}}. Then, dividing the above by $h/2$ and sending $h\to 0$ before summing over $1 \leq j \leq d$ yields
\begin{equation*}
\Tr\left((I_d \otimes u^\top )D   C(x)C(x)  C(x)^+   \right) + \Tr\left( v C(x) \right)\leq 0.
\end{equation*}
In view of \eqref{ourconddrift} and \eqref{eq: trace prod vec}, this shows that 
\begin{equation*}
\langle b(x), u \rangle + \frac{1}{2} \Tr(vC(x)) \leq \langle u, b(x)-\frac{1}{2} \sum_{j=1}^{d} D C^j(x) (CC^+)^j(x) \rangle \leq 0
\end{equation*}
for all $(u,v) \in \mathcal{N}^2_{\mathcal{D}}(x)$.
To conclude, it remains to observe that  $(D \phi(x),D ^2\phi(x)) \in \mathcal{N}^2_{\mathcal{D}}(x)$ whenever $\phi \in \mathcal{C}^2(\mathbb{R}^d,\mathbb{R})$  is such that $\displaystyle\max_{  \mathcal{D}} \phi=\phi(x) \geq 0$. Hence, $\mathcal{L}\phi(x) \leq 0$.  
\end{proof}

\begin{proposition}[Sufficient conditions of Theorem \ref{MainTheorem}]\label{propsufficiency}
	Under the assumptions of Theorem \ref{MainTheorem}, assume that conditions \eqref{ourcondcov} and \eqref{ourconddrift} hold for all $x \in \calD$ and  $u \in \mathcal{N}^{1}_{\mathcal{D}}(x)$. Then, $\calD$ is \textit{stochastically invariant} with respect to the diffusion \eqref{diffusionsdeinvariance}. 
\end{proposition}

\begin{proof} We already know from Proposition \ref{propmaximum} that $\mathcal{L}$ satisfies the  {positive maximum principle}. Then,  \cite[Theorem 4.5.4]{eth} yields the existence of a solution to the martingale problem associated to $\mathcal{L}$ with sample paths in the space of càdlàg functions with values in $\mathcal{D}^{\Delta}:=\mathcal{D} \cup \{\Delta\}$, the one-point compactification of $\mathcal{D}$. 
The discussion preceding \cite[Proposition 3.2]{CFY05} and \cite[Proposition 5.3.5]{eth}, recall our  linear growth conditions \eqref{growthconditions}, then  shows that the solution has a modification with continuous sample paths in $\mathcal{D}$. Finally, \cite[Theorem 5.3.3]{eth} implies the existence of a weak solution $(X,W)$ such that $X_t \in \mathcal{D}$ for all $t \geq 0$ almost surely. 

\end{proof}
 %%%%%%%%%%%%%%%%%%%%%%%%%%%%%%%%%%%%%%%%%%%
 %%%%%%%%%%%%%%%%%%%%%%%%%%%%%%%%%%%%%%%%%%%
 
\section{A  generic   application}\label{SectionExamples}
 
We show in this section how   Theorem \ref{MainTheorem} can be applied in various examples of application. We restrict to a two-dimensional setting for ease of computations and notations. 
\vs2

We first provide a generic tractable characterization for the stochastic invariance of all state spaces $\calD \subset \bbR^2$ of the following form:
\begin{equation}\label{eq:Dexample}
\calD=\{(\bar x,\widetilde x) \in \bbR^2, \bar x \in \calD_1 \mbox{ and } \phi(\bar x,\widetilde x) \in  \calD_2\},
\end{equation} 
where $\calD_1 \subset \bbR$ and $\calD_2 \subset \bbR$ are closed subsets and $\phi$ is a continuously differentiable function. 
 
Then, $\calD$ can be characterized through 
$
\Phi: (\bar x,\widetilde x) \mapsto (\bar x, \phi(\bar x,\widetilde x))
$
by 
 $$
 \calD= \Phi^{-1}(\calD_1 \times  \calD_2), 
 $$ 
 and \cite[Exercise 6.7 and Proposition 6.41]{roc} provides the following description of the normal cone whenever
\begin{equation}\label{eq: Hx}
\mbox{  $\Phi$ is differentiable at $x$ and   its Jacobian $D  \Phi(x)$ has full rank} \tag{$H_x$}
\end{equation}
 holds at any point $x\in \Dc$.
 
\begin{proposition}\label{prop:conenormalapplication} Fix $x=(\bar x,\widetilde x) \in \mathcal{D}$ such that  \reff{eq: Hx} holds. Then, 
	\begin{equation*}
	\calN^1_{\calD}(x) = \left \{  \left( \begin{array}{c} \bar{u} + \partial_1 \phi(x) \widetilde u     \\  \partial_2 \phi(x) \widetilde u  \end{array} \right), \bar u \in \calN^1_{ \calD_1}(\bar x) \mbox{ and } \widetilde u \in \calN^1_{ \calD_2}(\phi(\bar x, \widetilde x)) \right\},
	\end{equation*}
	in which $\partial_{i}\phi$ is the derivative with respect to the $i$-th component. 
\end{proposition}

When $x$ lies in the interior of $\calD$, $\Nc^1_{\Dc}(x)={\{0\}}$ and  \eqref{ourcondcov}-\eqref{ourconddrift} are trivially verified. Hence, it suffices to control  $b$ and $C$ on the boundary of the domain in order to ensure the stochastic invariance of $\Dc$ as stated by the following proposition, in which we use the notations 
\begin{equation}\label{eq: notations example}
b=(\bar b,\widetilde b )^{\top}, \;C=(C_{ij})_{ij}\;\mbox{ and } \;\partial_u=u_2 \partial_1  - u_1 \partial_2.
\end{equation}  

\begin{proposition}\label{prop:condbordapplication} Let $\calD$ be as in \eqref{eq:Dexample} and $x=(\bar{x}, \widetilde{x}) \in \partial \calD$ be such that \reff{eq: Hx} holds. Fix  $u=(u_1,u_2)^\top \in \calN^1_{\calD}(x)$   as in Proposition \ref{prop:conenormalapplication}. Under the assumptions of Theorem \ref{MainTheorem},  \eqref{ourcondcov}-\eqref{ourconddrift} are equivalent to the following:
	\begin{enumerate}[(a)]
		\item Either  $\widetilde u \neq 0$ and 
		\begin{subnumcases}{}
		C(x)= C_{11}(x)   \left( \begin{array}{cc} 1 & -\frac{u_1 }{u_2 }     \\ -\frac{u_1 }{u_2 }  & \frac{u_1^2 }{u_2^2 }   \end{array} \right), \label{ourcondcovapplication}
		\\
		\langle u, b(x) \rangle -\frac{\mathds{1}_{\{ C_{11}(x) \neq 0\}}}{2(u_1^2 + u_2^2)}  \left[ u_1u_2\partial_u (C_{11}-C_{22})(x)  + (u_2^2 - u_1^2) \partial_uC_{12}(x) \right] \leq 0. \label{ourconddriftapplication}
		\end{subnumcases}
		\item Or,  $\widetilde u = 0$, $u_1=\bar u$ and  
		\begin{subnumcases}{}
		C(x)\mathds{1}_{\{\bar u\ne 0\}}= C_{22}(x)   \left( \begin{array}{cc} 0 & 0    \\ 0  & 1   \end{array} \right)\mathds{1}_{\{\bar u\ne 0\}}, \label{ourcondcovapplicationb}
		\\
		\bar u \left(\bar b(x)  -\frac{\mathds{1}_{\{ C_{22}(x) \neq 0\}}}{2} \partial_2 C_{12}(x)\right) \leq 0. \label{ourconddriftapplicationb}
		\end{subnumcases}
	\end{enumerate}
	
\end{proposition}

\begin{proof} {Case (a), $\widetilde u \neq 0$:} Since  $D  \Phi(x)$ has full rank, $ \partial_2 \phi (x)\neq 0$ and therefore $u_2 \neq 0$.
	 	Since $C(x) \in \bbS^2$, \eqref{ourcondcov} is clearly equivalent to \eqref{ourcondcovapplication}.

	If $C_{11}(x) \neq 0$, \eqref{ourcondcovapplication} implies that  $u= ( \bar{u} + \partial_1 \phi(x) \widetilde u    ,  \partial_2 \phi(x) \widetilde u )^\top$ spans the kernel of $C(x)$. Therefore, by Proposition \ref{proppseudoinverseproj},
	\begin{equation*}
	C(x)C(x)^+  = I_2 - \frac{1}{\| u\|^2} u u^\top = \frac{1}{u_1^2+u_2^2} \left( \begin{array}{cc} u_2^2 & -u_1u_2    \\  -u_1u_2 & u_1^2  \end{array} \right).
	\end{equation*}

	Straightforward computations yield  
	\begin{equation*}
	\langle u, \sum_{j=1}^{2} DC^j(x) (CC^+)^j(x)  \rangle = \frac{1}{u_1^2+u_2^2} \left[ u_1u_2 \partial_u (C_{11}-C_{22})(x) + (u_2^2 - u_1^2) \partial_uC_{12}(x) \right],
	\end{equation*}
	recall the notations introduced in \reff{eq: notations example}. This shows the equivalence between  \eqref{ourconddrift} and \eqref{ourconddriftapplication} when $C_{11}(x) \neq 0$. 
	
	If $C_{11}(x)=0$, then \eqref{ourcondcovapplication} implies that $C(x)C(x)^+=0$ and \eqref{ourconddrift} reads $\langle u, b(x) \rangle \leq 0$.
	\vs2
	
	Case (b), $\widetilde u = 0$: If $\bar u =0$, then $u=0$ and there is nothing to prove. Otherwise,  $u_1=\bar u \neq 0$.

	Since $C(x) \in \bbS^2$, \eqref{ourcondcov} is clearly equivalent to $C_{11}(x)=0$ and $C_{21}(x)=C_{12}(x)=0$, that is \eqref{ourcondcovapplicationb}. 
	  If $C_{22}(x) \neq 0$, then  \eqref{ourcondcovapplicationb} provides
	\begin{equation*}
	C(x)C(x)^+  = \left( \begin{array}{cc} 0 & 0    \\  0 & 1  \end{array} \right),
	\end{equation*} 
	and straightforward computations yield
	\begin{equation*}
	\langle u, \sum_{j=1}^{2} DC^j(x) (CC^+)^j(x)  \rangle =  \bar u \partial_2 C_{12}(x),
	\end{equation*}
	which shows the equivalence between  \eqref{ourconddrift} and \eqref{ourconddriftapplicationb} when $C_{22}(x) \neq 0$. 	
	If $C_{22}(x)=0$, then $C(x)C(x)^+=0$ and \eqref{ourconddrift} reads $\bar  u  \bar b(x) \leq 0$.
\end{proof}

 Note that $\bar u=0$ when $\calD_1=\bbR$, which will be the case from now on. In the sequel, we impose more structure on the coefficients, as it is usually done  in the construction of invariant diffusions. This permits to deduce an explicit form of $(b,C)$ on the whole domain from the   boundary conditions \eqref{ourcondcovapplication}-\eqref{ourconddriftapplication}. As already stated, Theorem \ref{MainTheorem} can be directly applied to a large class of diffusions,  e.g.~affine 
diffusions \cite{dfs,film,sp12} and  {polynomial diffusions} \cite{fla,lar}, not only for closed subsets of $\mathbb{R}^d$, but even when $\mathcal{D} \subset \mathbb{S}^d$ (as in \cite{cuchf}) since $\mathbb{S}^d$ can be identified with $\bbR^{\frac{d(d+1)}{2}}$ by using the half-vectorization operator. We start by defining these two main structures.

\begin{definition}[Affine and polynomial diffusions]\label{def:affinepolynomial}
	$X$ is  a polynomial diffusion on $\calD$ if: 
	\begin{enumerate}[(i)]
		\item
		There  exist $\bar{b}^i, \widetilde{b}^i \in \mathbb{R}$, $0 \leq i \leq 2$, and $A^i \in \mathbb{S}^2$,  $1 \leq i \leq 5$, such that $b:x\mapsto b(x):=(\bar{b}(x),\widetilde{b}(x)) \in \bbR^2$ and $C: x \mapsto C(x) \in \mathbb{S}^2$  have the following form: 
		\begin{eqnarray}
		\left\{
		\begin{array}{rl}
		\bar b(x)&=\bar{b}^0 + \bar{b}^1 \bar{x} + \bar{b}^2 \widetilde x,  \\
		\widetilde b(x)&=\widetilde{b}^0 + \widetilde{b}^1 \bar{x} + \widetilde{b}^2 \widetilde x,   \\
		C(x)&=A^0 + A^1 \bar{x} + A^2 \widetilde{x} + A^3 \bar x^2 + A^4 \bar x \widetilde x+ A^5 \widetilde x^2, 
		\end{array}
		\right. \label{eq:Cpolynomial} 
		\end{eqnarray} 
		for all $x=(\bar x, \widetilde x) \in \calD$.
		\item
		$C(x) \in \bbS^d_+$, for all $x \in \calD$. 
	\end{enumerate}
	When $A^i=0$  for all $3 \leq i \leq 5$, we say that $X$ is  an affine diffusion.    
\end{definition}

Then, it is clear that $b$ and $C$ are $\mathcal{C}^\infty$ and satisfy the linear growth conditions \eqref{growthconditions}. \vs2

In what follows, we highlight the interplay between the geometry/curvature of the boundary and the coefficients $b$ and $C$. The three explicit examples below characterize the invariance for flat, convex and concave boundaries.  

\begin{example}[Canonical state space]\label{ex:canonical}
	Fix $\calD_1=\bbR$, $\calD_2=\bbR_+$ and $\phi(\bar x, \widetilde x )= \widetilde x$. Then $\calD=\bbR \times \bbR_+$ and $\calN^1_{\calD}(x)=\{0\} \times \mathds{1}_{\{\tilde x=0\}} \bbR_-$. Hence, \eqref{ourcondcovapplication}-\eqref{ourconddriftapplication} are  equivalent to
	\begin{equation*}
	C( \bar x, 0)	= C_{11}(\bar x, 0)   \left( \begin{array}{cc} 1 &  0     \\  0 & 0 \end{array}\right) \quad \mbox{ and } \quad 
	\widetilde b(\bar x,0) -\frac{\mathds{1}_{\{ C_{11}(\bar x,0) \neq 0\}}}{2}  \partial_1 C_{12} (\bar x, 0) \geq 0, \quad \mbox{ for all } \bar x  \in \bbR.
	\end{equation*}
	If we now impose the structural condition \reff{eq:Cpolynomial}, then straightforward computations lead to the characterization in \cite{dfs} for affine diffusions. The case of polynomial diffusions can be treated similarly. 
\end{example}

\begin{example}[Parabolic convex state space]\label{ex:convex}
	Let us consider the following parabolic state space:
	\begin{equation*}
	\calD=\{(\bar x,\widetilde x) \in \bbR^2, \widetilde x \geq \bar{x}^2 \}.
	\end{equation*} 
	
	Then, with the previous notations,  $\calD_1=\bbR$, $\calD_2 = \bbR_+$ and  $ \phi(\bar x,\widetilde x) = \widetilde x - \bar{x}^2$. Therefore, the first order normal cone given by Proposition \ref{prop:conenormalapplication} reads
	\begin{equation*}
	\calN^1_{\calD}(x) =  \left( \begin{array}{c} -2 \bar x     \\  1 \end{array}\right) \bbR_-, \quad \mbox{for all } x=(\bar x, \bar{x}^2) \in \partial \calD.
	\end{equation*}
	
	Conditions \eqref{ourcondcovapplication}-\eqref{ourconddriftapplication} are therefore equivalent to 
	
	\begin{subnumcases}{}
	C(x)= C_{11}(x)   \left( \begin{array}{cc} 1 & 2\bar x     \\ 2\bar x  & 4\bar x^2   \end{array} \right), \label{eq:convexparaboliccov}
	\\
	\langle u, b(x) \rangle -\frac{\mathds{1}_{\{ C_{11}(x) \neq 0\}}}{2(1+4\bar x^2)}  \left[ -2 \bar x\partial_u (C_{11}-C_{22})(x)  + (1 - 4\bar x^2) \partial_uC_{12}(x) \right] \geq 0, \label{eq:convexparabolicdrift}
	\end{subnumcases}
	for all $\bar x \in \bbR$, $x=(\bar x, \bar x^2)$ and $u=(-2\bar x,1)^\top$.

	If we now impose an additional affine structure on the diffusion $X=(\bar X, \widetilde X)$, as in  Duffie \textit{et al.}~\cite[Section 12.2]{dfs}, we recover   the characterization given in Gourieroux and Sufana \cite[Proposition 2]{gou}.  
Indeed, Proposition \ref{prop:condbordapplication} says that $\calD$ is invariant if and only if there exists $\alpha \geq 0$ such that 	
	\begin{enumerate}[{\rm (a)}]
		\item 
		\begin{equation}\label{goucov}
		C(x)= \alpha \left( \begin{array}{cc} 1  & 2  \bar{x}  \\  2  \bar{x} & 4 \widetilde{x}  \end{array} \right){,} \;\;\mbox{for all $x=(\bar{x},\widetilde{x})\in \Dc$,}
		\end{equation}
		\item $\bar{b}^2=0$ and
		\begin{eqnarray}
		\left\{\begin{array}{l}
		\widetilde{b}^2 > 2 \bar{b}^1 \quad \mbox{ and }  \quad (\widetilde{b}^1 -2  \bar{b}^0)^2 \le 4(\widetilde{b}^2  -2 \bar{b}^1 )(\widetilde{b}^0 -\alpha ) \\
		\mbox{or}  \\
		\widetilde{b}^2 = 2 \bar{b}^1, \quad   \widetilde{b}^1= 2 \bar{b}^0  \quad \mbox{ and } \quad \widetilde{b}^0 \geq \alpha  .
		\end{array}\right. \label{goudrift1}
		\end{eqnarray}
	\end{enumerate}

Let us detail the computations: 	
	(a) The covariance matrix $C(x) \in \mathbb{S}^2_+$ is of the form \eqref{eq:convexparaboliccov} on the boundary. Since $C$ is affine in $(\bar{x},\bar{x}^2)$, then necessarily $C_{11}(x)$ is constant (or else $C_{22}(x)$ would have at least a polynomial dependence of order 3 in $\bar{x}$). Therefore, there exists $\alpha$ such that $C(x)$ has the  form  \eqref{goucov} at $x=(\bar x, \bar x^2)$, in which $\alpha\ge 0$ to ensure that  $C(0) \in \mathbb{S}^2_+$. Finally,  $C$ needs to have the same form \eqref{goucov} on  the whole state space $\Dc$, since it is affine.
	
	(b) We now derive the form of the drift vector $b(x)=(\bar{b}(x),\widetilde{b}(x)) \in \mathbb{R}^2$ by using  \eqref{eq:convexparabolicdrift}.
	 From \eqref{goucov}, elementary computations show that condition \eqref{eq:convexparabolicdrift} is equivalent to
	\begin{equation*}
	-2 \bar{b}^2  \bar{x}^3 + (\widetilde{b}^2  -2 \bar{b}^1 )\bar{x}^2  + (\widetilde{b}^1 -2  \bar{b}^0 )\bar{x} +  \widetilde{b}^0 -\alpha  \geq 0,\quad \mbox{for all } \bar{x} \in  \mathbb{R},
	\end{equation*}
	which is equivalent to \eqref{goudrift1}, when $\alpha >0$.
	If $\alpha=0$, the same conclusion holds. 
	
	Conversely, \eqref{goucov} clearly implies  \eqref{ourcondcov} and (ii) of Definition \ref{def:affinepolynomial} since $\det(C(x))=4\alpha(\widetilde{x}-\bar{x}^2)\geq 0 $ and  $\widetilde{x} \geq 0$ for all $(\bar x,\widetilde x) \in \mathcal{D}$. Moreover, \eqref{goudrift1} leads to \eqref{ourconddriftapplication} by the same  computations as above.

\end{example}

\begin{example}[Parabolic concave state space]\label{ex:concave}
	We now consider the epigraph of the concave function $\bar x \mapsto -\bar x^2$, 
	\begin{equation*}
	\calD=\{(\bar x,\widetilde x) \in \bbR^2, \widetilde x \geq - \bar{x}^2 \}.
	\end{equation*} 
	
	It follows that  $\calD_1=\bbR$, $\calD_2 = \bbR_+$, $ \phi(\bar x,\widetilde x) = \widetilde x + \bar{x}^2$ and  
	\begin{equation*}
	\calN^1_{\calD}(x) =  \left( \begin{array}{c} 2 \bar x     \\  1 \end{array}\right) \bbR_-, \quad \mbox{for all } x=(\bar x, -\bar{x}^2) \in \partial \calD,
	\end{equation*}
	from Proposition \ref{prop:conenormalapplication}. Hence, conditions \eqref{ourcondcovapplication}-\eqref{ourconddriftapplication} are now equivalent to 
	
	\begin{subnumcases}{}
	C(x)= C_{11}(x)   \left( \begin{array}{cc} 1 & -2\bar x     \\ -2\bar x  & 4\bar x^2   \end{array} \right), \label{eq:concaveparaboliccov}
	\\
	\langle u, b(x) \rangle -\frac{\mathds{1}_{\{ C_{11}(x) \neq 0\}}}{2(4\bar x^2+1)}  \left[ 2 \bar x\partial_u (C_{11}-C_{22})(x)  + (1 - 4\bar x^2) \partial_uC_{12}(x) \right] \geq 0, \label{eq:concaveparabolicdrift}
	\end{subnumcases}
	for all $\bar x \in \bbR$, $x=(\bar x, -\bar x^2)$ and $u=(2\bar x,1)^\top\in -\calN^1_{\calD}(x)$.
	
	Let us first note that the above shows that we can not construct an affine diffusion living in $\Dc$, that is not degenerate, unless it lives on the boundary only. Indeed, if $C$ is affine then $C_{11}=:\alpha$  has to be constant, because of \eqref{eq:concaveparaboliccov},  and   $C$ is of the form \eqref{eq:concaveparaboliccov} with $(-\widetilde x)$ in place of $\bar x^2$. Since $C(x) \in \bbS^2_+$, we must have $\alpha \geq 0$ and  $\det C(x)=-4 \alpha^2 (\widetilde x + \bar x^2) \geq 0$. Thus, $\alpha=0$ unless we restrict to points $(\bar x,\widetilde x)$ on the boundary.  If we do so, it is not difficult to derive a necessary and sufficient condition on the coefficients  from the identity $\widetilde X=-\bar X^{2}$. 
	\vs2
	
	We now impose a polynomial structure on the  diffusion $X=(\bar X, \widetilde X)$, such that $\bar X$ is affine on its own, i.e.~$\bar b$ and $C_{11}$ are of affine form and only depend on $\bar x$. 	
	This extends \cite[Example 5.2]{lar} and entirely characterizes the stochastic invariance of $\Dc$ with respect to this structure of diffusion.
	By Proposition \ref{prop:conenormalapplication},  $\calD$ is invariant if and only if there exist $\alpha , \beta \geq 0$, such that	
	\begin{enumerate}[{\rm (a)}]
		\item  
		\begin{equation}\label{gouconcavecov}
		C(x)=  \left( \begin{array}{cc} \alpha & -2  \alpha\bar{x}  \\  -2 \alpha \bar{x} & (4 \alpha + \beta) \bar x^2 + \beta \widetilde{x}  \end{array} \right){,} \;\;\mbox{for all $x=(\bar{x},\widetilde{x})\in \Dc$,}
		\end{equation}
		\item  $\bar{b}^2=0$ and
		\begin{eqnarray}
		\left\{\begin{array}{l}
		\widetilde{b}^2 < 2 \bar{b}^1 \quad \mbox{ and }  \quad (\widetilde{b}^1 +2  \bar{b}^0)^2 \le 4(-\widetilde{b}^2  +2 \bar{b}^1 )(\widetilde{b}^0 +\alpha ) \\
		\mbox{or}  \\
		\widetilde{b}^2 = 2 \bar{b}^1, \quad   \widetilde{b}^1= -2 \bar{b}^0  \quad \mbox{ and } \quad \widetilde{b}^0 \geq -\alpha  
		\end{array}\right. .\label{gouconcavedrift}
		\end{eqnarray}
	\end{enumerate}

{Let us do the computations explicitly:} 	
	(a) The covariance matrix $C(x) \in \mathbb{S}^2_+$ is of the form \eqref{eq:concaveparaboliccov} on the boundary. Therefore, $C_{11}(x) \geq 0$, for all $x \in \partial\calD$. Since $C_{11}$ is affine and only depends on $\bar{x} \in \bbR$, then necessarily $C_{11}$ is a non-negative constant on the whole space $\calD$. Therefore, there exists $\alpha \geq 0$ such that $C_{11}(.)=\alpha$ on $\calD$. Moreover, \eqref{eq:Cpolynomial} reads on the boundary
	\begin{equation*}
	C(x)=A^0 + A^1 \bar x + (A^3-A^2) \bar x^2 - A^4 \bar x^3 + A^5 \bar x^4, \quad \mbox{ for all } \bar x \in \bbR.
	\end{equation*} 
	Therefore, comparing with  \eqref{eq:concaveparaboliccov} leads to    $A^4=A^5=0$ and the existence of   $\beta,\beta'$ such that $C$ is of the form \begin{equation*}
		   \left( \begin{array}{cc} \alpha & -2  \alpha\bar{x}   \\  -2 \alpha \bar{x}  & 4 \alpha  \bar x^2   \end{array} \right)
		   +
		     \left( \begin{array}{cc} 0 &  \beta' \\  \beta'&   \beta   \end{array} \right)(\widetilde x +\bar x^{2})
		\end{equation*}
 on the whole space $\Dc$. We now use the fact that $C(\Dc) \subset \bbS^2_+$. In particular, taking $\bar x=0$ shows that we must have $\alpha\beta\widetilde x-(\beta')^{2}\widetilde x^{2}\ge 0$ for all $\widetilde x\ge 0$, so that $\beta'=0$.  Similarly,     $4 \alpha \bar x^2+ \beta(\widetilde x +   \bar x^2 ) \geq 0$ must hold for all $x \in \Dc$, which is equivalent to   $\beta \geq 0$.  
 	
	(b) We now derive the form of the drift vector $b(x)=(\bar{b}(x),\widetilde{b}(x)) \in \mathbb{R}^2$ by using  \eqref{eq:concaveparabolicdrift}.
	 Since $X$ is affine on its own, $\bar b^2 =0$. From \eqref{gouconcavecov}, elementary computations show that condition \eqref{eq:concaveparabolicdrift} is equivalent to
	\begin{equation*}
	(-\widetilde{b}^2  +2 \bar{b}^1 )\bar{x}^2  + (\widetilde{b}^1 +2  \bar{b}^0 )\bar{x} +  \widetilde{b}^0 +\alpha  \geq 0,\quad \mbox{for all } \bar{x} \in  \mathbb{R},
	\end{equation*}
	which is equivalent to \eqref{gouconcavedrift}, when $\alpha >0$.
	If $\alpha=0$, the same conclusion holds. 
	
	Conversely, \eqref{gouconcavecov}-\eqref{gouconcavedrift} show that $X$ is a polynomial diffusion such that $\bar X$ is affine on its own since $\det(C(x))=\alpha \beta (\widetilde{x}+\bar{x}^2) \geq 0 $ and  $4 \alpha \bar x^2  + \beta ( \widetilde x + \bar x^2)\geq 4\alpha\bar x^2 \geq 0 $ for all $(\bar x,\widetilde x) \in \mathcal{D}$.  \eqref{gouconcavecov} clearly implies  \eqref{ourcondcov}.  Moreover, \eqref{gouconcavedrift} leads to \eqref{ourconddriftapplication} by the same  computations as above.

\end{example}

We conclude with a final remark on the interplay between the local geometry of the boundary, the coefficients $b$ and $C$ and the structure of the diffusion. 

\begin{remark}\label{rem:geometrycoeff}
	\begin{enumerate}[(i)]
		\item 
		Curvaceous boundary and covariance matrix: the curvature of the boundary plays a crucial role in determining the covariance structure. In Example \ref{ex:canonical}, the canonical state space, which shows no curvature, imposes strict constraints on the covariance matrix. Whereas, for curved domains, as in Examples \ref{ex:convex}-\ref{ex:concave}, the first order normal cone is a more complicated object and induces a richer covariance structure on the boundary.  
		\item
		Convexity and  drift direction: Figure \ref{fig:inwardoutward} visualizes the direction of the drift $b(0)=(\bar b^0, \widetilde b^0)$ with respect to the convexity of the domain. When the domain is convex, as in Example \ref{ex:convex}, the drift is necessarily inward pointing since $\widetilde b^0 \geq \alpha$, with $\alpha \geq 0$ from \eqref{goudrift1}. However, when the domain is concave, as in Example \ref{ex:concave}, the drift could even be outward pointing. This follows from the fact that $\widetilde b^0 \geq -\alpha$, with $\alpha \geq 0$ in \eqref{gouconcavedrift}. 
		%\item
		%Convexity and affine diffusions:  in Example \ref{ex:convex}, we characterized the invariance for affine diffusions while, in Example \ref{ex:concave}, we failed to construct non-degenerate affine diffusions. Necessarily  $\Dc \subset \{x, C(x) \in \bbS^d_+ \}$ which is convex when $C$ is affine. Hence, the maximal domain  on which we can expect to construct affine diffusions  is convex. This explains the compatibility of convex domains with affine diffusions, see \cite{cuchf,dfs,sp12}.
	\end{enumerate}
\end{remark}

\begin{figure}[h!]
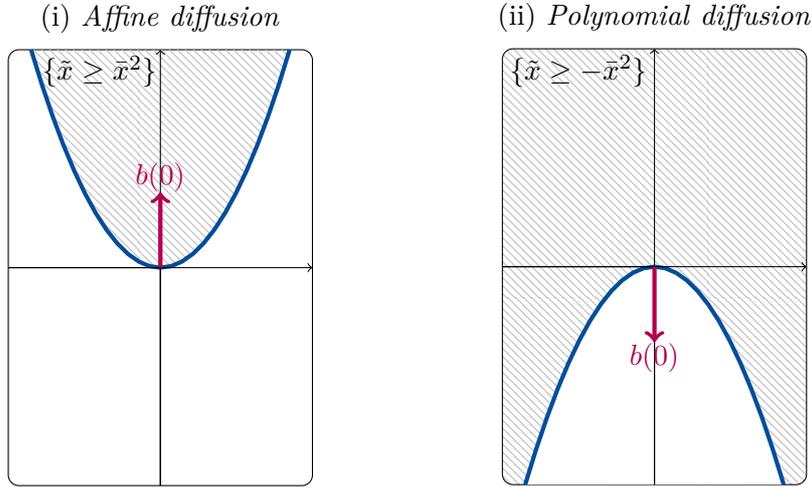

	\centering
	\begin{center}
		\inwardoutward
	\end{center}
	\rule{35em}{0.5pt}
	\caption[inwardoutward]{\textbf{Interplay between the convexity of the domain and the direction of the drift:} (i) Inward pointing drift for convex domains (Example \ref{ex:convex}). (ii)  Possible outward pointing drift for concave domains (Example \ref{ex:concave}).}
	\label{fig:inwardoutward}
\end{figure}

%%%%%%%%%%%%%%%%%%%%%%%%%%%%%%%%%%%%%%%%
\section{Additional remark on the boundary non-attainment}\label{sectionboundarynon}
In this last section, we provide a sufficient condition for  the stochastic invariance of the interior of $\Dc$,  when  $\Dc$  has a  smooth boundary. The result is a direct implication of \cite[Proposition 3.5]{sp12} derived with the help of \textit{McKean's argument} (see \cite[ Section 4]{mayp}). Moreover, we extend the tractable conditions of \cite[Proposition 3.7]{sp12} given for  {affine diffusions}. Our result could be easily used in the context of {polynomial  diffusions} for instance.  

\begin{proposition}\label{propboundarynon}
Let $\calD\subset \mathbb{R}^d$ be closed with a non-empty interior $\mathring{\calD}$ that is a maximal connected subset of $\{x, \Phi(x) <0 \}$ where $\Phi \in \calC^2(\mathbb{R}^d,\mathbb{R})$ such that $\partial\calD=\Phi^{-1}(0)$. Assume that $b$ and $C$ are continuous and satisfy assumptions \eqref{growthconditions}-\eqref{eq: extension C}. Moreover, assume that $C \in \mathcal{C}^{1}(\R^{d},\mathbb{S}^d_{+})$. Then $\mathring{\calD}$ is stochastically invariant if there exists $v \in \mathbb{R}^d$ such that
		\begin{subnumcases}{}
		D  \Phi(x)  C(x)  = \Phi(x) v^\top \label{ourcondcovopen}
		\\
		\langle  D  \Phi(x)  , b(x)-\frac{1}{2} \sum_{j=1}^{d} D C^j(x)e_j   \rangle \leq 0  \label{ourconddriftopen}
		\end{subnumcases}
		for all $x \in \mathcal{D}$.  
\end{proposition}

\begin{proof}
	Fix $x \in \mathring{\calD}$. By differentiating \eqref{ourcondcovopen} with the help of Propositions \ref{propdifferentiationrules} and \ref{propappendixformulas}, we obtain 
	\begin{equation*}
	  v D  \Phi(x)=( C(x) \otimes I_{1})D^{2}\Phi(x)+ ( I_{d}\otimes D\Phi(x))D C(x)=C(x) D^{2}\Phi(x)+ ( I_{d}\otimes D\Phi(x))D C(x),
	\end{equation*} 
	which, combined  with   \eqref{ourconddriftopen}, leads to  
	\begin{eqnarray*}
	\langle D  \Phi(x),b(x) \rangle &\leq& \frac{1}{2} \Tr\left[ (I_d \otimes D  \Phi(x) ^\top ) D C(x) \right] \\
	 &=& -\frac{1}{2} \Tr\left[C(x)  D^2 \Phi(x)  \right] + \frac{1}{2}  D  \Phi(x) v\\
	 &=& -\frac{1}{2} \Tr\left[ C(x)  D^2 \Phi(x)  \right] + \frac{1}{2} \Phi(x)^{-1} D  \Phi(x) C(x)D  \Phi(x)^{\top}.
	\end{eqnarray*}
	We conclude by using \cite[Proposition 3.5]{sp12} (after a change of the sign, since  $\calD$ is assumed to be a connected subset of $\{x,\Phi(x)>0\}$  in \cite[Proposition 3.5]{sp12}).
\end{proof}

\begin{example}
\begin{enumerate}[{\rm (i)}]
\item \emph{Square root process:} Let us consider again the process defined by $dX_t=b(X_t)dt + \eta \sqrt{X_t}dW_t$, for some $\eta>0$, on $\calD=\bbR_+$. Then, $\Phi:x\mapsto-x$ and   \eqref{ourcondcovopen}-\eqref{ourconddriftopen} are equivalent to $v=\eta^2$ and $b(0) \geq \frac{\eta^2}{2}$. These are the well known conditions for the boundary non-attainment of the square-root process.
\item \emph{Affine diffusions:} More generally, let $\calD\subset\bbR^d$ satisfy  the assumptions of Proposition  \ref{propboundarynon} and take $C(x)=A^0+\sum_{j=1}^d A^jx^{j}$ for some $A^j \in \bbS^d$, $1 \leq j \leq d$. Then differentiating $C$ shows that condition \eqref{ourconddriftopen} is equivalent to $\langle  D  \Phi(x)  , b(x)-\frac{1}{2} \sum_{j=1}^{d} (A^j)^j  \rangle \leq 0$ yielding \cite[Proposition 3.7]{sp12}.
\item \emph{Jacobi diffusion:} Set $\calD=(0,1]$ and consider a {polynomial  diffusion} $X$ on $\cal{D}$, \textit{i.e.}~$b$ is affine  and $C$ is a polynomial  of degree two. Theorem \ref{MainTheorem} applied on $[0,1]$ immediately yields that de dynamics of $X$ must be of the form $dX_t=\kappa(\theta-X_t)dt+\eta \sqrt{X_t(1-X_t)}dW_t$ where $\kappa , \eta \geq 0$ and   $0\leq \theta \leq 1$. Now a localized version of Proposition  \ref{propboundarynon} shows that $\calD=(0,1]$ is stochastically invariant under the additional condition that $\kappa\theta \geq \frac{\eta^2}{2}$.
\end{enumerate}
\end{example}

Proposition \ref{propboundarynon} is important in practice since it gives, in many cases, the existence and the uniqueness of a global \textit{strong solution} to   \eqref{diffusionsdeinvariance}  as discussed in the following remark. 

\begin{remark}
Let $\calD$ be as in Proposition \ref{propboundarynon}. Assume that $C \in \mathcal{C}^2(\mathring{\Dc},\bbS^d_{+})$   and that $b$ is locally Lipschitz (which is clearly the case for affine and polynomial diffusions). By  \cite[Remark 1 page 131]{fri},  $\sigma=C^{\frac12}$ is locally Lipschitz on $\mathring{\calD}$. Therefore, when the boundary is never attained,   \eqref{diffusionsdeinvariance} starting from any element $x \in \mathring{\calD}$ admits a global \textit{strong solution} and pathwise-uniqueness holds.   
\end{remark}

\begin{appendices}

\section{Matrix tools}\label{AppendixMatrix}
 For the reader's convenience, we collect in this Appendix some definitions and properties of matrix tools intensively used in the proofs throughout the article. For a complete review and proofs we refer to \cite{mag80,mag88,neu}.
 \vs2
 
 We start by recalling the definition of the Moore-Penrose pseudoinverse which generalizes the concept of invertibility of square matrices, to non-singular and non-square matrices. In the following, we denote by $\M^{m,n}$ the collection of $m\times n$ matrices. 

\begin{definition}[Moore-Penrose pseudoinverse]\label{defpseudoinverse}
Fix $A\in \M^{m,n}$. The Moore-Penrose pseudoinverse of $A$ is the unique $n \times m $ matrix $A^+$ satisfying: $AA^+A=A$, $A^+AA^+=A^+$, $AA^{+}$ and $A^{+}A$ are Hermitian. 
\end{definition}

\begin{proposition}\label{proppseudoinversespectral}
If $A\in \M^{d}$ has the spectral decomposition $Q\Lambda Q^{\top}$ for some orthogonal matrix $Q\in \M^{d}$ and a diagonal matrix $\Lambda=\diag{(\lambda_{i})_{i\le d}}\in \M^{d}$. Then, $A^{+}= Q\Lambda^{+} Q^{\top}$ in which $\Lambda^{+}=\diag{(\lambda_{i}^{-1}\mathds{1}_{\{\lambda_{i}\ne 0\}})_{i\le d}}$, and $AA^{+}= Q \diag{(\mathds{1}_{\{\lambda_{i}\ne 0\}})_{i\le d}}Q^{\top}$. If moreover $A$ is positive semi-definite and $B=A^{\frac12}$, then $B^{+}=Q(\Lambda^{+})^{\frac12} Q^{\top}$. 
\end{proposition}

\begin{proposition}\label{proppseudoinverseproj}
If $A\in \M^{m,n}$, then $AA^+$ is the orthogonal projection on the image of $A$. 
\end{proposition}

 We now collect some useful identities  on the Kronecker product.

\begin{definition}[Kronecker product]\label{def: kronecker}
Let $A=(a_{ij})_{i\le m_{1},j\le n_{1}}\in \M^{m_{1},n_{1}}$   and $B\in \M^{m_{2},n_{2}}$. The Kronecker product $(A\otimes B)$ is defined as the $m_1 m_2 \times n_1 n_2$ matrix  
\begin{equation*}
A\otimes B = \left( \begin{array}{ccc} a_{11} B & \cdots  & a_{1n_1} B \\ \vdots &  &  \vdots \\ a_{m_1 1} B & \cdots  & a_{m_1 n_1} B \end{array} \right).
\end{equation*}
\end{definition}

\begin{proposition}\label{propappendixformulas} Let $A$ and $B$ be as in Definition \ref{def: kronecker},   $C\in \M^{n_{1},n_{3}}$  and $D\in \M^{n_{2},n_{4}}$.  Then, 
\begin{eqnarray*}
&(A\otimes B)(C\otimes D )=(AC\otimes BD),&\\
& {A\otimes B}=A(I_{n_{1}} \otimes B) \mbox{ if $m_{2}=1$},&\\
& {A\otimes B}=B(A \otimes I_{n_{2}}) \mbox{ if $m_{1}=1$.}&\\
\end{eqnarray*}
\end{proposition}

The following definitions extend the concept {of   Jacobian} matrix and show how to nicely stack the partial derivatives of a matrix-valued function $F: \M^{n, q} \mapsto \M^{m ,p}$  by using the vectorization operator (see \cite[Chapter 9]{mag88}).

\begin{definition}[Vectorization  operator]
Let $A\in \M^{m,n}$. The vectorization operator $\vecop$  transforms the matrix into a vector in $\mathbb{R}^{mn}$ by stacking all the columns of the matrix $A$ one underneath the other. 
\end{definition}

\begin{definition}[Jacobian matrix]\label{defjacobianmatrix}
Let $F$ be a differentiable map from $\M^{n,q}$ to $\M^{m,p}$. The Jacobian matrix $D F(X)$ of $F$ at $X$ is defined as the following $mp \times nq$ matrix:
\begin{equation*}
D F(X)=\frac{\partial \vecop(F(X)) }{\partial \vecop(X)^\top}. 
\end{equation*}
\end{definition}

\begin{proposition}[Product rule]\label{propdifferentiationrules}
 Let $G$ be a differentiable map from $\M^{n,q}$ to $\M^{m,p}$ and  $H$  be a differentiable map from $\M^{n,q}$ to $\M^{p,l}$. 
Then,   $D(GH) =(H^\top \otimes I_{m})DG +(I_{l}\otimes G  )D H $.
  \end{proposition}

 \end{appendices}

\bibliographystyle{unsrtnat} % Use the "unsrtnat" BibTeX style for formatting the Bibliography

\end{document}